\documentclass[a4paper,11pt]{article}

\usepackage[a4paper]{geometry}

\usepackage{authblk}

\usepackage{amsmath}
\usepackage{amsthm}
\usepackage{amssymb}
\usepackage{amsfonts}
\usepackage{mathtools}
\usepackage{bm}
\usepackage{bussproofs}

\usepackage{cancel}
\usepackage{lscape}

\usepackage{float}
\usepackage{booktabs}
\usepackage{caption}
\usepackage{adjustbox}

\usepackage{graphicx}
\usepackage{tikz}
\usetikzlibrary{cd,backgrounds,positioning,trees,shapes,arrows,patterns,topaths,calc}
\usepackage[dvipsnames]{xcolor}

\usepackage[
  backend=biber,
  natbib=true,
  maxcitenames=9,
  maxbibnames=99,
  abbreviate=false,
  eprint=false,
  doi=true,
  giveninits=true,
  casechanger=expl3,
]{biblatex}
\DeclareFieldFormat{doi}{\url{https://doi.org/#1}}
\DeclareFieldFormat{url}{\url{#1}}
\usepackage{hyperref}
\hypersetup{
  hidelinks,
  linktoc = all,
  breaklinks = true,
}
\usepackage[nameinlink]{cleveref}

\newtheorem{theorem}{Theorem}
\newtheorem{lemma}[theorem]{Lemma}
\newtheorem{corollary}[theorem]{Corollary}

\newtheorem{proposition}[theorem]{Proposition}

\theoremstyle{definition}
\newtheorem{definition}[theorem]{Definition}
\newtheorem{remark}[theorem]{Remark}
\newtheorem{example}[theorem]{Example}
\newtheorem{notation}[theorem]{Notation}

\theoremstyle{remark}

\numberwithin{theorem}{section}

\newcommand{\alg}[1]{\mathfrak{#1}}
    \newcommand{\A}{\alg{A}}
    \newcommand{\B}{\alg{B}}

\renewcommand{\sp}[1]{\mathfrak{#1}}
    \newcommand{\X}{\sp{X}}
    \newcommand{\Y}{\sp{Y}}

\newcommand{\fr}[1]{\mathcal{#1}}
    \newcommand{\F}{\fr{F}}

\newcommand{\class}[1]{\mathcal{#1}}
    \newcommand{\V}{\class{V}}
    \newcommand{\U}{\class{U}}
    
    \newcommand{\D}{\class{D}}
    
    \newcommand{\C}{\class{C}}
    \newcommand{\R}{\class{R}}
    \renewcommand{\S}{\class{S}}

\newcommand{\fsi}{_\mathrm{fsi}}

\newcommand{\pow}[1]{\mathcal{P} (#1)}
\newcommand{\Log}{\mathsf{Log}}

\newcommand{\logic}[1]{\mathsf{#1}}

    \newcommand{\K}{\logic{K}}
    \newcommand{\Kf}{\logic{K4}}
    \newcommand{\Kff}[2]{\logic{K4^{#1}_{#2}}}
    \newcommand{\Sf}{\logic{S4}}
    \newcommand{\wKf}{\logic{wK4}}
    \newcommand{\IPC}{\logic{IPC}}

\newcommand{\NExt}[1]{\mathsf{NExt}{#1}}

\newcommand{\Sub}{\mathsf{Sub}}

\renewcommand{\phi}{\varphi}
\newcommand{\emp}{\emptyset}

\newcommand{\goodbox}{\hspace{.2ex}\text{%
  \tikz[baseline=-.6ex, rounded corners=.01ex, line width=.12ex]
    {\draw(-.6ex,-.6ex) rectangle (.6ex,.6ex);}}\kern.2ex}
\newcommand{\gooddiamond}{\hspace{.2ex}\text{%
  \tikz[baseline=-.6ex, rounded corners=.01ex, rotate=45, line width=.12ex]
    {\draw(-.5ex,-.5ex) rectangle (.5ex,.5ex);}}\kern.2ex}
\renewcommand{\Box}{\goodbox}
\renewcommand{\Diamond}{\gooddiamond}

\newcommand{\Dia}{\Diamond}

\newcommand{\Boxx}[1]{\Box^{#1}}

\newcommand{\up}{{\uparrow}}

\newcommand{\inj}{\hookrightarrow}
\newcommand{\surj}{\twoheadrightarrow}

\newcommand{\Lor}{\bigvee}
\newcommand{\Land}{\bigwedge}

\newenvironment{acknowledgements}{%
  \begin{abstract}
}{%
  \end{abstract}
}

\renewcommand{\S}{\mathcal{S}}

\begin{document}

\title{Stable Canonical Rules and Formulas for \\ Pre-transitive Logics via Definable Filtration}
\author{Tenyo Takahashi\footnote{\href{mailto:t.takahashi@uva.nl}{t.takahashi@uva.nl}}}
\affil{\small Institute for Logic, Language and Computation, \\ University of Amsterdam}
\date{}
\maketitle

\begin{abstract}
We generalize the theory of stable canonical rules by adopting definable filtration, a generalization of the method of filtration. We show that for a modal rule system or a modal logic that admits definable filtration, each extension is axiomatizable by stable canonical rules. Moreover, we provide an algebraic presentation of Gabbay's filtration and generalize stable canonical formulas and the axiomatization results via stable canonical formulas for $\Kf$ to pre-transitive logics $\Kff{m+1}{1} = \K + \Dia^{m+1} p \to \Dia p$ $(m \geq 1)$. As consequences, we obtain the finite model property of $\Kff{m+1}{1}$-stable logics and a characterization of splitting and union-splitting logics in the lattice $\NExt{\Kff{m+1}{1}}$. There are continuum many $\Kff{m+1}{1}$-stable logics that are neither $\Kf$-stable logics nor subframe logics. Finally, we introduce $m$-stable canonical formulas, strengthening the axiomatization results for these logics.\footnote{This paper is based on \cite[Chapter 3]{TakahashiThesis}.}
\end{abstract}

\section{Introduction}

One of the most powerful tools in the study of the lattice of modal logics and the lattice of superintuitionistic logics is \emph{characteristic formulas}. Characteristic formulas, also called \emph{algebra-based formulas} or \emph{frame-based formulas}, refer to variations of formulas that are defined from finite algebras or finite relational structures (e.g., finite Kripke frames) so that their validity has a semantic characterization. Thus, logics axiomatized by characteristic formulas are often easier to analyze. We refer to \cite{JankovFormulasAxiomatization2022b} for a comprehensive overview of the techniques of characterization formulas for superintuitionistic logics. The first type of characteristic formulas, \emph{Jankov formulas} or \emph{Jankov-de Jongh formulas}, was introduced and studied by Jankov \cite{jankov1963} and independently by de Jongh \cite{dejongh1968} for superintuitionistic logics. Jankov formulas were used in \cite{jankov1963} to construct continuum many superintuitionistic logics. Their modal logic analog, \emph{Fine formulas} or \emph{Jankov-Fine formulas}, was introduced by Fine \cite{fineK4I}. A generalization to $n$-transitive modal logics was constructed by Rautenberg \cite{rautenbergSplittingLatticesLogics1980} using algebraic methods. Further development includes \emph{subframe formulas}, \emph{cofinal subframe formulas}, and \emph{canonical formulas} (see, e.g., \cite[Chapter 9]{czModalLogic1997}). A remarkable feature of canonical formulas is that they axiomatize all transitive logics \cite{zakharyaschevCanonicalFormulas1}, which makes them particularly useful. For various applications of the technique of canonical formulas, see \cite[Chapters 9-16]{czModalLogic1997}. 

Recently, Je\v r\'abek \cite{Jeřábek_2009} generalized the idea of canonical formulas to inference rules and defined \emph{canonical rules}. These rules are based on a version of \emph{selective filtration} (see, e.g., \cite[Section 9.3]{czModalLogic1997}) and axiomatize all rule systems over $\Kf$. Je\v r\'abek \cite{Jeřábek_2009} applied canonical rules to construct admissible bases and to provide an alternative proof of the decidability of the admissibility problem for transitive logics (such as $\Kf$, $\Sf$, and $\logic{GL}$) as well as the intuitionistic logic. 

Bezhanishvili et al. \cite{stablecanonicalrules}, motivated by the method of \emph{filtration} (see, e.g., \cite[Section 2.3]{blackburnModalLogic2001} and \cite[Section 5.3]{czModalLogic1997}), defined \emph{stable canonical rules} and showed that they axiomatize all rule systems over the least modal rule system $\S_\K$. They also introduced \emph{stable canonical formulas} for $\Kf$ as an alternative to canonical formulas. Applying stable canonical rules and formulas, \cite{stablemodallogic} showed the finite model property (the fmp) of $\Kf$-stable logics (see \Cref{3: Def M-stable logic}), and \cite{admissiblebases} provided yet another proof of the decidability of the admissibility problem for transitive logics; more applications include \cite{BezhanishviliGhilardi2014,BLOKESAKIATHEOREMSSTABLE2025a}.

However, the situation is much trickier if we drop transitivity, even in the pre-transitive setting. Recall that \emph{$m$-transitivity} refers to the property defined by $\Dia^{m+1} p \to \Dia^m p \lor \cdots \lor p$ for $m \geq 1$ and \emph{pre-transitivity} refers to them as a whole. Pre-transitive logics are natural targets when trying to extend studies from transitive logics to non-transitive ones, for the clear similarity between the properties. The decidability of the admissibility problem in pre-transitive logics is open, and is a long-standing open problem for $\K$ \cite[Problem 16.4]{czModalLogic1997}. No variant of subframe formulas is known for pre-trasitive logics as pointed out in \cite[Section 1]{TwoTypesFiltrations2025}, and neither are stable canonical formulas. A main challenge here is that, while selective filtration is useful mostly for transitive logics and standard filtration is also constructed for some non-transitive logics such as $\K$, they do not easily apply to pre-transitive logics. In fact, the fmp for pre-transitive logics in general is another long-standing open problem \cite[Problem 11.2]{czModalLogic1997}.

In this paper, we generalize the theory of stable canonical rules and formulas by adopting \emph{definable filtration}, a generalization of standard filtration that appeared earlier in \cite{ageneralfiltrationmethod1972} and was explicitly introduced in \cite{Kikot2020}. Recall that the method of filtration uses a finite subformula-closed set $\Theta$ to transform a Kripke model into a finite one, by identifying points that agree on all formulas in $\Theta$ and imposing requirements on the accessibility relation according to $\Theta$. Definable filtration, on the other hand, can use a larger set $\Theta' \supseteq \Theta$ to identify points so that the equivalence relation is finer, while the requirements on the accessibility relation are still determined by $\Theta$. A typical example of definable filtration is Gabbay's filtration \cite{ageneralfiltrationmethod1972} for the pre-transitive logics $\Kff{m+1}{1} = \K + \Dia^{m+1} p \to \Dia p$ $(m \geq 1)$. We will provide an algebraic proof that these logics admit definable filtration in \Cref{Sec 4}. Note that it is open if they admit standard filtration; see \Cref{Rem mismatch} for why standard approaches would fail.

Using this definable filtration, as a first step towards characteristic formulas for non-transitive logics, we generalize stable canonical formulas and the axiomatization result for $\Kf$ \cite{stablecanonicalrules} to the logics $\Kff{m+1}{1}$. This is achieved by a generalization of the axiomatization results via stable canonical rules, allowing the base rule system or logic to be any rule system or logic that admits definable filtration. These generalizations yield the following applications:
\begin{itemize}
    \item We observe that the assumption for \emph{$\logic{M}$-stable logics} to have the fmp in the result in \cite{stablemodallogic}, namely, admitting (standard) filtration, can be weakened to admitting definable filtration. This allows us to obtain the fmp for $\Kff{m+1}{1}$-stable logics, which include continuum many logics that are neither $\Kf$-stable logics nor subframe logics (see \Cref{Thm conti many K4m1 stable}). Thus, these logics form new large classes of non-transitive logics that have the fmp.
    \item We obtain an axiomatic characterization of splitting and union-splitting logics in the lattice $\NExt{\Kff{m+1}{1}}$. 
    \item Another application is in \cite{takahashi2025choppingfinelyfinitecountermodels} (see also \cite[Chapter 4]{TakahashiThesis}), where a combinatorial method is introduced to prove the fmp for large classes of logics and rule systems. This method makes an essential use of the axiomatization results via stable canonical rules, reducing the problem of constructing finite countermodels to arbitrary formulas to that of constructing finite countermodels to stable canonical rules.
\end{itemize}

Finally, we present a proper subclass of the class of stable canonical formulas (up to equivalence) that axiomatizes all extensions of $\Kff{m+1}{1}$. These formulas are called \emph{$m$-stable canonical formulas} and encode the structure of finite modal algebras or finite modal spaces up to $m$ steps, which better matches the nature of pre-transitive logics. We show that they are sufficient for axiomatizing all extensions of $\Kff{m+1}{1}$.

This paper is organized as follows. \Cref{Sec 2} reviews the necessary preliminaries, including stable canonical rules. \Cref{Sec 3} discusses definable filtrations from an algebraic perspective and their connection to the theory of stable canonical rules. \Cref{Sec 4} presents an algebraic formulation of Gabbay's filtration for pre-transitive logics $\Kff{m+1}{1}$. Based on this, \Cref{Sec 5} develops the theory of stable canonical formulas for $\Kff{m+1}{1}$. \Cref{Sec 6} defines $m$-stable canonical formulas and shows the axiomatization result via them. Finally, \Cref{Sec 7} outlines potential directions for future work.

\section{Preliminaries} \label{Sec 2}

We assume familiarity with modal logic. We refer to \cite{blackburnModalLogic2001, czModalLogic1997} for basic syntax and Kripke semantics. We will only work with normal modal logics in this paper, so we simply call them logic. For a logic $L$, $\NExt{L}$ denotes the lattice of normal extensions of $L$. We will use $\Box^m \phi$ as an abbreviation of $\Box \cdots \Box \phi$ with $m$ many $\Box$ and $\Boxx{\leq m}$ as an abbreviation of $\phi \land \cdots \land \Box^m \phi$; the dual abbreviation applies to $\Dia$.

Recall that a \emph{modal algebra} is a pair $\A = (A, \Dia)$ of a Boolean algebra $A$ and a unary operation $\Dia$ on $A$ such that $\Dia 0 = 0$ and $\Dia (a \lor b) = \Dia a \lor \Dia b$. Valuations, satisfaction, and validity are defined as usual. The above abbreviation applies to operations on modal algebras as well. Recall that a \emph{modal space} is a pair $\X = (X, R)$ of a Stone space and a binary relation $R \subseteq X \times X$ such that $R[x]$ is closed for each $x \in X$ and the set $R^{-1}[U]$ is clopen for each clopen set $U \subseteq X$. Modal spaces are also known as \emph{descriptive frames}. Finite modal spaces are identified with finite Kripke frames. Continuous p-morphisms between modal spaces are simply called p-morphisms. It is well known that the category of modal algebras and homomorphisms is dually equivalent to the category of modal spaces and p-morphisms. When drawing modal spaces, we adopt the convention that $\bullet$ denotes an irreflexive point and $\circ$ denotes a reflexive point.

We sketch the construction of the dual equivalence functors, which is built on top of \emph{Stone duality} (see, e.g., \cite[Section 5.4]{blackburnModalLogic2001}). Given a modal algebra $\A$, the dual modal space of $\A$ is $\A_* = (A_*, R_{\Dia})$ where 
\begin{enumerate}
    \item $A_*$ is the dual Stone space of $A$: $A_*$ as a set is the set of all ultrafilters on $A$, and the topology is generated by the base $\{\beta(a): a \in A\}$, where the function $\beta: A \to \pow{A_*}$ is defined by $\beta(a) = \{x \in A_*: a \in x\}$,
    \item $R_{\Dia}$ is defined by 
    \[x R_{\Dia} y \text{ iff for any $a \in A$, $a \in y$ implies $\Dia a \in x$}. \]
\end{enumerate} 
Given a homomorphism $h: \A \to \B$, the dual p-morphism of $h$ is $h_*: \B_* \to \A_*; x \mapsto h^{-1}[x]$. Conversely, given a modal space $\X = (X, R)$, the dual modal algebra of $\X$ is $\X^* = (X^*,  \Dia_R)$ where 
\begin{enumerate}
    \item $X^*$ is the Boolean algebra of clopen subsets of $X$,
    \item $\Dia_R \: a = R^{-1}[a]$ for $a \in A$.
\end{enumerate}
Given a p-morphism $f: \X \to \Y$, the dual homomorphism is $f^*: \Y^* \to \X^*; a \to f^{-1}[a]$. 

We assume familiarity with basic universal algebra and refer to \cite{UniversalAlgebraFundamentals2011, ACourseInUniversalAlgebra1981} for an introduction. For a class $\class{K}$ of modal algebras, $\V(\class{K})$ (resp. $\U(\class{K})$) denotes the least variety (resp. universal class) containing $\class{K}$. The following characterization of subdirectly irreducible (s.i. for short) modal algebras is from Rautenberg \cite{rautenbergSplittingLatticesLogics1980}. Note that an opremum may not be unique. 

\begin{proposition} \label{2: Prop opremum}
    A modal algebra $\A$ is s.i.~iff $\A$ has an \emph{opremum}, that is, an element $c \in \A$ such that $c \neq 1$ and for any $a \neq 1$, there is $n \in \omega$ such that $\Boxx{\leq n} a \leq c$. 
\end{proposition}

\subsubsection*{Stable homomorphisms and the closed domain condition}

Stable homomorphisms and the closed domain condition for modal algebras and modal spaces were introduced in \cite{stablecanonicalrules}, generalizing that for Heyting algebras and Priestley spaces introduced in \cite{LocallyFiniteReducts2017}. Intuitively, a stable homomorphism $h$ does not preserve $\Dia$ (so it is not a modal algebra homomorphism), but it does satisfy the inequality $\Dia h(a) \leq h(\Dia a)$; the closed domain condition indicates for which elements $a$ the equality $\Dia h(a) = h(\Dia a)$ holds. Stable homomorphisms were studied in \cite{TopocanonicalCompletionsClosure2008} as semihomomorphisms and in \cite{Ghilardi01012010} as continuous morphisms. 

\begin{definition}
    Let $\A$ and $\B$ be modal algebras. 
    \begin{enumerate}
        \item A Boolean homomorphism $h: A \to B$ is a \emph{stable homomorphism} if $\Dia h(a) \leq h(\Dia a)$ for all $a \in A$.
        \item Let $h: A \to B$ be a stable homomorphism. For $a \in A$, we say that $h$ satisfies the \emph{closed domain condition (CDC) for $a$} if $h(\Dia a) = \Dia h(a)$. For $D \subseteq A$, we say that $h$ satisfies the \emph{closed domain condition (CDC) for $D$} if $h$ satisfies CDC for all $a \in D$.
    \end{enumerate}
\end{definition}

These definitions dualize to modal spaces as follows.

\begin{definition}
    Let $\X = (X, R)$ and $\Y = (Y, Q)$ be modal spaces. 
    \begin{enumerate}
        \item A continuous map $f: X \to Y$ is a \emph{stable map} if $xRy$ implies $f(x)Qf(y)$ for all $x, y \in X$.
        \item Let $f: X \to Y$ be a stable map. For a clopen subset $D \subseteq Y$, we say that $f$ satisfies the \emph{closed domain condition (CDC) for $D$} if
\[
Q[f(x)] \cap D \neq \emptyset \Rightarrow f(R[x]) \cap D \neq \emptyset
\]
        for all $x \in \X$.
    For a set $\D$ of clopen subsets of $Y$, we say that $f: X \to Y$ satisfies \emph{the closed domain condition (CDC) for $\D$} if $f$ satisfies CDC for all $D \in \D$.
    \end{enumerate}
    
\end{definition}

It follows directly from the definition that a stable homomorphism $h: A \to B$ satisfying CDC for $A$ is a modal algebra homomorphism, and a stable map $f: X \to Y$ satisfying CDC for $\pow{Y}$ is a p-morphism.

\begin{notation}
    We write $h: \A \inj_D \B$ if $h$ is a stable embedding satisfying CDC for $D$, and $\A \inj_D \B$ if there is such an $h$. We write $f: \X \surj_\D \Y$ if $f$ is a surjective stable map satisfying CDC for $\D$, and $\X \surj_\D \Y$ if there is such an $f$.
\end{notation}

\subsubsection*{Stable canonical rules}
We refer to \cite{stablecanonicalrules, Jeřábek_2009, kracht8ModalConsequence2007} for inference rules and consequence relations in modal logic. A \emph{modal multi-conclusion rule} $\rho$ is an expression of the form
    \begin{prooftree}
        \AxiomC{$\Gamma$}
        \UnaryInfC{$\Delta$}
    \end{prooftree}
(or $\Gamma/\Delta$) where $\Gamma$ and $\Delta$ are finite sets of formulas. If $\Delta$ is a singleton, $\rho$ is called \emph{single-conclusion}. If $\Gamma = \emp$, $\rho$ is called \emph{assumption-free}. We will only address modal multi-conclusion rules, so we simply call them \emph{rules}. A single-conclusion assumption-free rule $/\phi$ can be identified with the formula $\phi$. 

A \emph{normal modal multi-conclusion consequence relation} or \emph{normal modal multi-conclusion rule system} is a set $\S$ of rules satisfying:
    \begin{enumerate}
        \item $\phi/\phi \in \S$,
        \item $\phi, \phi \to \psi / \psi \in \S$,
        \item $\phi / \Box \phi \in \S$,
        \item $/\phi \in \S$ for all $\phi \in \K$,
        \item $\Gamma/\Delta \in \S$ implies $\Gamma, \Gamma' / \Delta, \Delta' \in \S$ \quad (\emph{weakening}),
        \item $\Gamma/\Delta, \phi \in \S$ and $\Gamma, \phi/\Delta \in \S$ implies $\Gamma/ \Delta \in \S$ \quad (\emph{cut}),
        \item $\Gamma/\Delta \in \S$ implies $\sigma(\Gamma)/\sigma(\Delta) \in \S$ for any substitution $\sigma$ \quad (\emph{substitution}).
    \end{enumerate}
We simply call normal modal multi-conclusion consequence relations \emph{rule systems}. For a set $\R$ of rules, let $\S_0 + \R$ be the least rule system containing $\S_0 \cup \R$. For a logic $L$, let $\S_L$ be the least rule system containing $L$. When $\S = \S_0 + \R$, we say that $\S$ is \emph{axiomatized} by $\R$ over $\S_0$. When $\S_L$ is axiomatized by $\R$ over $\S_{L_0}$, we say that $L$ is axiomatized by $\R$ over $L_0$. When $\S + \rho = \S + \rho'$, we say that $\rho$ and $\rho'$ are \emph{equivalent} over $\S$. 

A modal algebra $\A$ \emph{validates} a rule $\Gamma/\Delta$, written $\A \models \Gamma/\Delta$, if for any valuation $V$ on $\A$, $V(\gamma) = 1$ for all $\gamma \in \Gamma$ implies $V(\delta) = 1$ for some $\delta \in \Delta$. Recall that logics correspond one-to-one to varieties and we write $\V(L) \coloneq \{\A: \A \models L\}$ for a logic $L$ and $\Log(\V) \coloneq \{\phi: \V \models \phi\}$ for a variety $\V$ (see, e.g., \cite[Section 7.6]{czModalLogic1997}). In a similar manner, each rule corresponds to a universal sentence and each rule system to a universal class. We write $\U(\S) \coloneq \{\A: \A \models \S\}$ for a rule system $\S$ and $\S(\U) \coloneq \{\rho: \U \models \rho\}$ for a universal class $\U$. For a logic $L$, let $\Sigma(L)$ be the least rule system containing $\{/\phi: \phi \in L\}$, that is, $\Sigma(L) = \S_\K + \{/\phi: \phi \in L\}$. Note that $\Sigma(L) = \S_L$. For a rule system $\S$, let $\Lambda(\S) = \{\phi: /\phi \in \S\}$. Then, $\Sigma$ and $\Lambda$ preserve the subset relation, and we have $\Lambda(\Sigma(L)) = L$ for a logic $L$ and $\Sigma(\Lambda(\S)) \subseteq \S$ for a rule system $\S$. So, if a logic $L$ is axiomatized by $\R$ over $L_0$, then $L = \Sigma(\S)$ for a rule system $\S$ axiomatized by $\R$ over $\S_{L_0}$.

\emph{Stable canonical rules} are introduced in \cite{stablecanonicalrules} as an alternative to canonical rules, the theory of which is developed in \cite{Jeřábek_2009}. Both of them are generalizations of characteristic formulas to multi-conclusion rules. They are defined from finite modal algebras or finite modal spaces, and their validity can be fully characterized by the algebras or spaces. While canonical rules, generalizing Zakharyaschev's canonical formulas \cite{zakharyaschevCanonicalFormulas1} (see also \cite[Chapter 9]{czModalLogic1997}), use selective filtration (see, e.g., \cite[Section 5.5]{czModalLogic1997}), stable canonical rules use the standard filtration (see, e.g., \cite[Section 2.3]{blackburnModalLogic2001} and \cite[Section 5.3]{czModalLogic1997}). A special feature of stable canonical rules is that, contrary to canonical rules, any rule system can be axiomatized by stable canonical rules over the least normal modal rule system $\S_\K$.

The basic idea of stable canonical rules, as well as other characteristic formulas, is very similar to \emph{diagrams} widely used in model theory: to encode the structure of finite algebras or finite spaces (frames), but only partially. 

\begin{definition} \label{3: Def scr}
    Let $\A$ be a finite modal algebra and $D \subseteq A$. The \emph{stable canonical rule} $\rho(\A, D)$ associated to $\A$ and $D$ is the rule $\Gamma/\Delta$, where:
    \begin{align*}
        \Gamma  = & \{p_a \lor p_b \leftrightarrow p_{a \lor b} : a,b \in A\} \cup \\
        & \{\lnot p_a \leftrightarrow p_{\lnot a} : a \in A\} \cup \\
        & \{\Dia p_a \rightarrow p_{\Dia a} : a \in A\} \cup \\
        & \{p_{\Dia a} \rightarrow \Dia p_a : a \in D\},
    \end{align*}
    and
    \[\Delta = \{p_a : a \in A, a \ne 1\}.\]
\end{definition}

Stable canonical rules can also be defined directly via finite modal spaces (i.e., finite Kripke frames). 

\begin{definition} \label{3: Def scr sp}
    Let $\X = (X, R)$ be a finite modal space and $\D \subseteq \mathcal{P}(X)$. Define the \emph{stable canonical rule} $\rho(\X, \D)$ associated to $\X$ and $\D$ as the rule $\Gamma/\Delta$, where:
    \begin{align*}
        \Gamma = & \{\bigvee \{p_x : x \in X\}\} \cup \\
        & \{p_x \to \lnot p_y : x,y \in X, x \neq y\} \cup \\
        & \{p_x \to \lnot\Dia p_y : x,y \in X, x \cancel{R} y\} \cup \\
        & \{p_x \to \bigvee \{\Dia p_y : y \in D\} : x \in X, D \in \D, x \in R^{-1}[D]\},
    \end{align*}
and
\[\Delta = \{\lnot p_x : x \in X\}.\]
\end{definition}

The following semantic characterization and axiomatization results are from \cite{stablecanonicalrules}.

\begin{theorem} \label{3: Thm sc rule char}
    Let $\A$ be a finite modal algebra, $D \subseteq A$, and $\B$ be a modal algebra. Then 
    \[\B \not\models \rho(\A, D) \text{ iff } \A \inj_D \B.\]
\end{theorem}

For convenience, we provide the dual presentation of this fact. For a finite modal space $\F$, $\D \subseteq \pow{F}$, and a modal space $\X$, we have 
\[\X \not\models \rho(\F, \D) \text{ iff } \X \surj_\D \F.\]

\begin{theorem} \label{3: Thm scr ax over K} \leavevmode
    \begin{enumerate}
        \item Every rule system $\S$ is axiomatizable over $\S_\K$ by stable canonical rules. Moreover, if $\S$ is finitely axiomatizable over $\S_\K$, then $\S$ is axiomatizable over $\S_\K$ by finitely many stable canonical rules.
        \item Every logic $L$ is axiomatizable over $\K$ by stable canonical rules. Moreover, if $L$ is finitely axiomatizable over $\K$, then $L$ is axiomatizable over $\K$ by finitely many stable canonical rules.
    \end{enumerate}    
\end{theorem}

\section{Definable filtrations and stable canonical rules} \label{Sec 3}

In this section, we identify definable filtration, a generalized type of filtration, as a key concept for developing the theory of stable canonical rules. We generalize the axiomatization result over $\S_\K$ and $\K$ (\cite{stablecanonicalrules}) to any rule system or logic that admits definable filtration. 

We assume familiarity with the filtration method widely used in modal logic. See, e.g., \cite[Definition 2.36]{blackburnModalLogic2001} or \cite[Section 5.3]{czModalLogic1997} for the standard filtration for Kripke frames and models, and \cite[Section 4]{stablecanonicalrules} for an algebraic account. Recall that in standard filtration, given a subformula-closed set $\Theta$ of formulas, the equivalence relation $\sim_\Theta$ identifying points that agree on all formulas in $\Theta$ is used to quotient the original model. Applying standard filtration with a finite set $\Theta$ to a countermodel yields a finite countermodel, but the validity of the base logic may not be preserved under this construction. To avoid this, one can use a finer relation as long as the number of equivalence classes remains finite for quotienting. A first example of this kind of generalization appeared in \cite{ageneralfiltrationmethod1972}, which used the equivalence relation $\sim_{\Theta'}$ for some $\Theta' \supseteq \Theta$. This is a typical example of definable filtrations introduced below. More general versions of filtrations that use finer equivalence relations not necessarily induced by sets of formulas were used in \cite{Shehtman1987,Shehtman1993,GabbayShehtman1998}. See also \cite{vanbenthemModernFacesFiltration2023} for an overview and discussion of various notions of filtration.

For our purposes, even though we allow finer equivalence relations, we have to restrict ourselves to those induced by sets of formulas. This restriction ensures that the quotient map from a modal space to its filtration is continuous and enables an algebraic presentation, which will be crucial for developing the theory of stable canonical rules and formulas in the subsequent sections. This leads to the concept of \emph{definable filtration}, which was used earlier in \cite{ageneralfiltrationmethod1972} and explicitly defined in \cite{Kikot2020} for Kripke frames. Let us recall the definition of definable filtrations presented for modal spaces.

\begin{definition} \label{3: Def sp filtration}
    Let $\X = (X, R)$ be a modal space, $V$ be a valuation on $\X$, $\Theta$ be a finite subformula-closed set of formulas, and $\Theta'$ be a finite subformula-closed set of formulas containing $\Theta$. A \emph{definable filtration of $(\X, V)$ for $\Theta$ through $\Theta'$} is a modal space $\X' = (X', R')$ with a valuation $V'$ such that:
    \begin{enumerate}
        \item $X' = X/{\sim_{\Theta'}}$, where 
        \[x \sim_{\Theta'} y \text{ iff } (\X, V, x \models \phi \iff \X, V, y \models \phi \text{ for all $\phi \in \Theta'$}),\]
        \item $V'(p) = \{[x]_{\Theta'}: x \in V(p)\}$ for $p \in \Theta'$ and $V'(p) = \emp$ for $p \notin \Theta'$,
        \item $xRy$ implies $[x]_{\Theta'} R' [y]_{\Theta'}$,
        \item if $[x]_{\Theta'} R' [y]_{\Theta'}$ then ($y \models \phi$ implies $x \models \Dia \phi$ for $\Dia \phi \in \Theta$).
    \end{enumerate}
    We also call $\X'$ a \emph{definable filtration of $\X$ for $\Theta$ through $\Theta'$}.
\end{definition}

We drop the subscript $\Theta'$ when it is clear from the context. A standard filtration through $\Theta$ is always a definable filtration by taking $\Theta' = \Theta$. However, a definable filtration for $\Theta$ through $\Theta'$ is different from a filtration through $\Theta'$, as in the former, the extended set $\Theta'$ of formulas is used to construct the quotient $X'$, while the fourth requirement on the relation $R'$ is determined by $\Theta$. This distinction is important, as we will see in \Cref{Rem mismatch}.

Each equivalence class of $\sim_{\Theta'}$ can be represented as 
\[\bigcap \{V(\phi)^*: \phi \in \Theta'\},\]
where each $V(\phi)^*$ is either $V(\phi)$ or $V(\phi)^c$, and thus is clopen in $\X$. It follows that the quotient map $X \surj X/\sim_{\Theta'}$ is continuous. This allows an algebraic definition of definable filtrations.

\begin{definition} \label{3: Def alg filtration}
    Let $\A = (A, \Dia)$ be a modal algebra, $V$ be a valuation on $\X$, $\Theta$ be a finite subformula-closed set of formulas, and $\Theta'$ be a finite subformula-closed set of formulas containing $\Theta$. A \emph{definable filtration of $(\A, V)$ for $\Theta$ through $\Theta'$} is a modal algebra $\A' = (A', \Dia')$ with a valuation $V'$ such that:
    \begin{enumerate}
        \item $A'$ is the Boolean subalgebra of $A$ generated by $V[\Theta'] \subseteq A$, 
        \item $V'(p) = V(p)$ for $p \in \Theta'$ and $V'(p) = 0$ for $p \notin \Theta'$,
        \item The inclusion $\A' \inj \A$ is a stable homomorphism satisfying CDC for $D$, where 
        \[D = \{V(\phi) : \Dia \phi \in \Theta\}.\]
    \end{enumerate}
    We also call $\A'$ a \emph{definable filtration of $\A$ for $\Theta$ through $\Theta'$}.
\end{definition}

Note that the extended set $\Theta'$ of formulas is used to generate the Boolean subalgebra $A'$, while the closed domain $D$ is determined by $\Theta$. It is straightforward to verify that the two definitions of definable filtrations are dual to each other in the following sense, generalizing the proof for standard filtrations in \cite[Theorem 4.2]{stablecanonicalrules}. We will mostly use algebraic filtrations throughout the paper, while one can always reformulate it using the language of modal spaces.

\begin{proposition}
    Let $\Theta, \Theta'$ be finite subformula-closed sets of formulas such that $\Theta \subseteq \Theta'$. For modal algebras $\A, \A'$ with the dual spacec $\X, \X'$ and valuations $V_\A, V_{\A'}$ with the dual valuations $V_\X, V_{\X'}$, $(\A', V_{\A'})$ is a definable filtration of $(\A, V_{\A'})$ for $\Theta$ through $\Theta'$ iff $(\X', V_{\X'})$ is a definable filtration of $(\X, V_{\A'})$ for $\Theta$ through $\Theta'$.
\end{proposition}

Below are some well-known examples of standard filtration. A definable filtration that is not a standard filtration will be presented in \Cref{Sec 4}. 

\begin{example}[\cite{stablecanonicalrules}] \label{3: Ex filtration} \leavevmode
    \begin{enumerate}
        \item Recall that the \emph{least filtration} and the \emph{greatest filtration} are standard filtrations defined frame-theoretically by 
        \[[x]R^l[y] \text{ iff } (x \sim x' \text{ and } y \sim y' \text{ and } x'Ry', \text{ for some $x', y' \in X$})\]
        and 
        \[[x]R^g[y] \text{ iff } (y \models \phi \text{ implies } x \models \Dia \phi, \text{ for all $\Dia \phi \in \Theta$})\]
        respectively. The algebraic constructions of them are
        \[\Dia^l a = \Land \{ b \in A' : \Dia a \leq b \}\]
        and
        \[\Dia^g a = \Land \{ \Dia b : a \leq b \text{ and } b \in D^\lor \}\]
        respectively, where $D^\lor$ is the $(\lor, 0)$-subsemilattice of $\A'$ generated by $D$.
        
        \item Recall that the \emph{Lemmon filtration} (see, e.g., \cite[Section 2.3]{blackburnModalLogic2001} or \cite[Section 5.3]{czModalLogic1997}) is a standard filtration defined frame-theoretically by 
        \[[x] R^L [y] \text{ iff } (y \models \Dia^{\leq 1} \phi \text{ implies } x \models \Dia \phi, \text{ for all $\Dia \phi \in \Theta$}),\]
        where $\Dia^{\leq 1} a = a \lor \Dia a$. The algebraic construction of the Lemmon filtration is 
        \[\Dia^L a = \Land \{\Dia b : \Dia a \leq \Dia b \text{ and } \Dia^{\leq 1} a \leq \Dia^{\leq 1} b \text{ and } b \in D^\lor \}.\]
        
    \end{enumerate}
    
\end{example}

The importance of the filtration method is summarized in the following lemma. Note that the statement may not hold for a formula in $\Theta'{\setminus}\Theta$. See \Cref{Rem mismatch} for more discussion on the importance of the ``mismatch'' of $\Theta'$ and $\Theta$ in the definition of definable filtrations.

\begin{lemma}[Filtration Lemma] \label{3: Lem filtration}
    Let $(\A', V')$ be a definable filtration of $(\A, V)$ for $\Theta$ through $\Theta'$. Then $V(\phi) = V'(\phi)$ for all $\phi \in \Theta$. 
\end{lemma}

\begin{proof}
    By a routine induction on the complexity of $\phi$. Use the fact that $\Theta \subseteq \Theta'$ for the base case, that $A'$ is a Boolean subalgebra of $A$ for the Boolean cases, and that the inclusion is stable and satisfies CDC for $D$ for the $\Dia$ case.
\end{proof}

Using the Filtration Lemma, we can automatically deduce the finite model property (fmp) for logics and rule systems that \emph{admit definable filtration}.

\begin{definition} \label{3: Def admit filtration} \leavevmode
    \begin{enumerate}
        \item A class $\C$ of modal algebras \emph{admits definable filtration} if for any finite subformula-closed set $\Theta$ of formulas, there is a finite subformula-closed set $\Theta'$ containing $\Theta$ such that, for any modal algebra $\A \in \C$ and any valuation $V$ on $\A$, there is a definable filtration $(\A', V')$ of $(\A, V)$ for $\Theta$ through $\Theta'$ such that $\A' \in \C$.
        \item  A modal logic $L$ \emph{admits definable filtration} if the variety $\V(L)$ admits definable filtration. 
        \item A rule system $\S$ \emph{admits definable filtration} if the universal class $\U(\S)$ admits definable filtration. 
    \end{enumerate}
   
\end{definition}

Note that a logic $L$ admits definable filtration iff the rule system $\S_L$ admits definable filtration because they correspond to the same class of modal algebras.

\begin{remark}
    The notion of admitting filtration has been used in the literature with various meanings and strengths. An explicit definition of \emph{admitting filtration in the weak sense} and \emph{admitting filtration in the strong sense} is provided in \cite{stablemodallogic}, where the relation of the two definitions is also discussed. A similar discussion with Kripke frames and models can be found in \cite{Kikot2020}. We define it as \Cref{3: Def admit filtration} because this is just enough to prove \Cref{3: Thm scr complete} and \Cref{3: Cor scr complete}, though it is stronger than what is needed to prove \Cref{3: Prop filtration implies fmp}.
\end{remark}

\begin{proposition} \label{3: Prop filtration implies fmp}
    If a logic or a rule system admits definable filtration, then it has the fmp.
\end{proposition}

\begin{proof}
    We only show the statement for rule systems, and the case for logics follows similarly. Let $\S$ be a rule system that admits definable filtration. For any rule $\rho = \Gamma/\Delta \notin \S$, there is an $\S$-algebra $\A$ such that $\A \not\models \rho$, witnessed by some valuation $V$ on $\A$. By the assumption that $\S$ admits definable filtration, since $\A \in \U(\S)$, there is a finite subformula-closed set $\Theta'$ containing $\Sub(\Gamma \cup \Delta)$ and a definable filtration $(\A', V')$ of $(\A, V)$ for $\Sub(\Gamma \cup \Delta)$ through $\Theta'$ such that $\A' \in \U(\S)$. By \Cref{3: Lem filtration}, $\A, V \not \models \Gamma/\Delta$ implies $\A', V' \not \models \Gamma/\Delta$. Thus, $\A'$ is a finite $\S$-algebra that refutes $\rho$. Hence, $\S$ has the fmp. 
\end{proof}

Many logics are known to admit standard filtration and thus admit definable filtration. For example, $\K$, $\logic{T}$, and $\logic{D}$ admit the least and the greatest filtration. As for transitive logics, $\Kf$ and $\Sf$ admit the Lemmon filtration.

Now we prove the main theorems of this section, generalizing \Cref{3: Thm scr ax over K} by extending the base rule system (or logic) from $\S_\K$ (or $\K$) to any one that admits definable filtration. Let us recall the proof idea of \Cref{3: Thm scr ax over K} in terms of modal spaces. For any formula $\phi$, if a modal space $\X \not\models \phi$, then by applying standard filtration through $\Theta = \Sub(\phi)$, we obtain a finite modal space $\F \not\models \phi$ with a closed domain $\D \subseteq\pow{F}$. It follows from the definition of standard filtration that the quotient map from $\X$ to $\F$ is a stable map satisfying CDC for $\D$, so $\X \not\models \rho(\F, \D)$. We can show that there are only finitely many such pairs $(\F, \D)$ obtained this way, and $\phi$ is equivalent to the stable canonical rules of those pairs, which concludes the proof. If we try to generalize the proof with non-standard filtrations, then the quotient map may not be continuous, and the proof would fail. The definability of filtrations is meant to guarantee that the equivalence classes are clopen and that the induced quotient map is continuous. Moreover, this is a natural constraint because the original idea of \emph{general frames} (a variant of modal spaces) is to restrict valuations to the sets defined by formulas \cite[Section 8.1]{czModalLogic1997}. These sets are called \emph{admissible} and correspond to clopen sets in modal spaces. Thus, we can identify admitting definable filtration as a key feature of $\S_\K$ that was used in the original proof of \Cref{3: Thm scr ax over K}, allowing us to generalize the result. The proofs presented here are essentially the same as the original ones.

\begin{theorem} \label{3: Thm scr complete} 
    Let $\S$ be a rule system that admits definable filtration. For any rule $\rho$, there exist stable canonical rules $\rho(\A_1, D_1), \dots, \rho(\A_n, D_n)$ where each $\A_i$ is a finite $\S$-algebra and $D_i \subseteq A_i$, such that for any $\S$-algebra $\B$,
    \[\B \models \rho \text{ iff } \B \models\rho(\A_1, D_1), \dots, \rho(\A_n, D_n),\]
\end{theorem}

\begin{proof}
    Let $\rho = \Gamma/\Delta$ be a rule. If $\rho \in \S_L$, then we take $n=0$. Assume that $\rho \notin \S_L$ and let $\Theta = \Sub(\Gamma \cup \Delta)$. By the assumption that $\S$ admits definable filtration, there is a finite subformula-closed set $\Theta'$ containing $\Theta$ such that, for any $\S$-algebra $\B$ and any valuation $V$ on $\B$, there is a definable filtration $(\B', V')$ of $(\B, V)$ for $\Theta$ through $\Theta'$ such that $\B' \models \S$. Let $m = |\Theta'|$. Since Boolean algebras are locally finite, up to isomorphism, there are finitely many tuples $(\A, V, D)$ satisfying the following conditions:
    \begin{enumerate}
        \item $\A$ is a finite $\S$-algebra based on an at most $m$-generated Boolean algebra and $\A \not\models \rho$,
        \item $V$ is a valuation on $\A$ such that $\A, V \not\models \rho$ and $V(p) = 0$ for $p \notin \Theta'$,
        \item $D = \{V(\psi): \Dia\psi \in \Theta\}$.
    \end{enumerate}
    Let $(\A_1, V_1, D_1), \dots, (\A_n, V_n, D_n)$ be an enumeration of such tuples. We show that for any $\S$-algebra $\B$, $\B \models \rho$ iff $ \B \models\rho(\A_1, D_1), \dots, \rho(\A_n, D_n)$.

    Suppose that $\B \not\models \rho(\A_i, D_i)$ for some $1 \leq i \leq n$. Then there is a stable embedding $h: \A_i \inj_{D_i} \B$ by \Cref{3: Thm sc rule char}. Define a valuation $V$ on $\B$ by $V(p) = h \circ V_i (p)$. Since $h$ satisfies CDC for $D_i$, we have $V(\phi) = h \circ V_i (\phi)$ for all $\phi \in \Theta$. Therefore, since $\A_i, V_i \not\models \rho$ by the enumeration, $V_i(\gamma) = 1$ for all $\gamma \in \Gamma$ and $V_i(\delta) \neq 1$ for all $\delta \in \Delta$, and it follows that $V(\gamma) = 1$ for all $\gamma \in \Gamma$ and $V(\delta) \neq 1$ for all $\delta \in \Delta$, namely, $\B, V \not\models \rho$.

    Conversely, suppose that $\B \not\models \rho$. Let $V$ be a valuation on $\B$ such that $\B, V \not\models \rho$ and $V(p) = 0$ for $p \notin \Theta'$. Then, there is a definable filtration $(\B', V')$ of $(\B, V)$ for $\Theta$ through $\Theta'$ such that $\B' \models \S$. By the definition of definable filtrations, $B'$ is a Boolean subalgebra of $B$ generated by $V[\Theta']$, so $B'$ as a Boolean algebra is at most $m$-generated. By \Cref{3: Lem filtration}, since $\Gamma \cup \Delta \subseteq \Theta$ and $\B, V \not\models \rho$, we obtain $\B', V' \not\models \rho$. Let $D = \{V(\psi): \Dia\psi \in \Theta\}$. Then, the tuple $(\B', V', D)$ is identical to $(\A_i, V_i, D_i)$ for some $1 \leq i \leq n$. Since $\B' \inj_D \B$ by the definition of definable filtration, we have $\A_i \inj_{D_i} \B$, namely, $\B \not\models \rho(\A_i, D_i)$ by \Cref{3: Thm sc rule char}. Therefore, we conclude that $\B \models \rho$ iff $ \B \models\rho(\A_1, D_1), \dots, \rho(\A_n, D_n)$.
\end{proof}

It follows directly from \Cref{3: Thm scr complete} that any formula $\phi$ is also semantically equivalent to finitely many stable canonical rules over any rule system that admits definable filtration, by identifying $\phi$ with the rule $/\phi$. Thus, we obtain the following axiomatization result.

\begin{theorem} \label{3: Cor scr complete} \leavevmode
    \begin{enumerate}
        \item Let $\S$ be a rule system that admits definable filtration. Every rule system $\S' \supseteq \S$ is axiomatizable over $\S$ by stable canonical rules. Moreover, if $\S'$ is finitely axiomatizable over $\S$, then $\S'$ is axiomatizable over $\S$ by finitely many stable canonical rules.
        \item Let $L$ be a logic that admits definable filtration. Every logic $L' \supseteq L$ is axiomatizable over $L$ by stable canonical rules. Moreover, if $L'$ is finitely axiomatizable over $L$, then $L'$ is axiomatizable over $L$ by finitely many stable canonical rules.
    \end{enumerate}
\end{theorem}

\begin{proof} \leavevmode
    \begin{enumerate}
        \item Let $\S$ be a rule system that admits definable filtration. For any rule system $\S' \supseteq \S$, $\S' = \S + \{\rho_i: i \in I\}$ for a set $\{\rho_i: i \in I\}$ of rules. By \Cref{3: Thm scr complete}, each rule $\rho_i$ is semantically equivalent to a finite set of stable canonical rules $\{\rho(\A_{ij}, D_{ij}): 1 \leq j \leq n_i\}$ for $\S$-algebras. So, for any $\S$-algebra $\B$, $\B \models \S'$ iff $\B \models \rho(\A_{ij}, D_{ij})$ for all $i \in I$ and $1 \leq j \leq n_i$. Thus, $\S' = \S + \{\rho(\A_{ij}, D_{ij}): i \in I, 1 \leq j \leq n_i\}$. Moreover, if $\S'$ is finitely axiomatizable over $\S$, then we can choose $I$ to be finite, hence the set $\{\rho(\A_{ij}, D_{ij}): i \in I, 1 \leq j \leq n_i\}$ is also finite.

        \item Let $L$ be a logic that admits definable filtration. For any logic $L' \supseteq L$, $L' = L + \{\phi_i: i \in I\}$ for a set $\{\phi_i: i \in I\}$ of formulas. So, $\Sigma(L') = \Sigma(L) + \{/\phi_i: i \in I\}$. By (1), there is a set $\{\rho(\A_j, D_j): j \in J\}$ of stable canonical rules such that $\Sigma(L') = \Sigma(L) + \{\rho(\A_j, D_j): j \in J\}$. Moreover, if $L'$ is finitely axiomatizable over $L$, then we can choose $I$ to be finite, so that $J$ can also be chosen to be finite by (1).
    \end{enumerate}
\end{proof}

As we mentioned, many logics, such as $\K$, $\logic{T}$, $\logic{D}$, $\Kf$, and $\Sf$ admit standard filtration. We will see in the next section that pre-transitive logics $\Kff{m+1}{1} = \K + \Dia^{m+1} p \to \Dia p$ ($m \geq 1)$ admit definable filtration (\Cref{3: Thm K4m1 admits filtration}), while it is unknown if they admit standard filtration (see \Cref{Rem mismatch}). Thus, these results apply to them and their corresponding rule systems.

\begin{remark} \label{3: Rem scr computable}
    If the base rule system $\S$ is decidable, then the result of \Cref{3: Cor scr complete} is computable; that is, given a finite axiomatization of $\S'$ over $\S$, a finite set of stable canonical rules that axiomatize $\S'$ over $\S$ can be computed. This can be shown by observing that if $\S$ is decidable, then in the proof of \Cref{3: Thm scr complete}, the enumeration of tuples $(\A, V, D)$ is computable, so the finite set $\rho(\A_1, D_1), \dots, \rho(\A_n, D_n)$ of stable canonical rules that is equivalent to the given rule $\rho$ is also computable. This holds as well if the base logic is decidable. Since a finitely axiomatizable rule system with the fmp is decidable, this is the case for all logics we mentioned above. 
\end{remark}

\section{Gabbay's filtration} \label{Sec 4}

In this section, we give an algebraic proof of the fact that a certain type of pre-transitive logics admits definable filtration. Recall that pre-transitive logics $\Kff{m}{n}$ ($n<m$) are logics axomatized by $\Dia^m p \to \Dia^n p$ (or equivalently, $\Box^n p \to \Box^m p$) over $\K$. The logic $\Kff{m}{n}$ defines the condition 
\[\forall x \forall y \: (x R^m y \to x R^n y),\]
where $R^k$ is the $k$-time composition of $R$, on modal spaces. Note that $\Kff{2}{1}$ is the transitive logic $\Kf$. 

Despite the success of Lemmon filtration (a standard filtration) for $\Kf$, the fmp of pre-transitive logics in general, such as $\Kf^3_2$, is a long-standing open problem (e.g., \cite[Problem 11.2]{czModalLogic1997}), let alone definable filtration. Gabbay \cite{ageneralfiltrationmethod1972} provided a definable filtration for $\Kff{m+1}{1}$ ($m \geq 1$). However, as far as we know, there is no standard filtration constructed for these logics. Let us illustrate the difficulty in the following remark.

\begin{remark} \label{Rem mismatch}
    Let us try to construct a standard filtration for $\Kff{3}{1}$. Recall that the Lemmon filtration for $\Kf$ was defined by
    \[[x] R^L [y] \text{ iff } (y \models \phi \lor \Dia \phi \text{ implies } x \models \Dia \phi, \text{ for all $\Dia \phi \in \Theta$}).\]
    Lemmon filtration fails for $\Kff{3}{1}$ because the resulting quotient map may not be stable (i.e., order-preserving), thus the construction is no longer a filtration. The following example illustrates this failure, where $\Theta = \{p, \Dia p\}$, $xRy$, and $\lnot ([x] R^L [y])$.
    \begin{figure}[h]
        \centering
        
        \begin{tikzpicture}[scale=1]
            \node (x) at (0,0) {\( \bullet \)};
            \node (xx) at (-0.4,0) {\( x \)};
            \node (xxx) at (0.7,0) {\( \Dia \Dia p \)};
            \node (y) at (0,1) {\( \bullet \)};
            \node (yy) at (-0.4,1) {\( y \)};
            \node (yyy) at (0.5,1) {\( \Dia p \)};
            \node (z) at (0,2) {\( \bullet \)};
            \node (zz) at (-0.4,2) {\( z \)};
            \node (zzz) at (0.4,2) {\(p \)};
            \draw[->] (x) -- (y);
            \draw[->] (y) -- (z);

            \node (x1) at (6,2) {\( \circ \)};
            \node (xx1) at (6.4,2) {\( [x] \)};
            \node (y1) at (5,1) {\( \circ \)};
            \node (yy1) at (5-0.4,1) {\( [y] \)};
            \node (z1) at (5,2) {\( \bullet \)};
            \node (zz1) at (5-0.4,2) {\( [z] \)};
            \draw[->] (y1) -- (z1);
            \draw[->] (y1) -- (x1);
            \draw[->] (z1) -- (x1);
        \end{tikzpicture}
    \end{figure}
    This happens because of the ``$\Dia \phi$'' part in the premise ``$y \models \phi \lor \Dia \phi$''. Indeed, suppose that $xRy$. Then, if $y \models \phi \lor \Dia \phi$ for some $\Dia \phi \in \Theta$, we have $x \models \Dia \phi \lor \Dia \Dia \phi$, but we cannot apply the fact that $x \models \Dia \Dia \Dia p \to \Dia p$ to deduce $x \models \Dia p$, which was possible if $x \models \Kf$. In other words, we cannot conclude $[x] R^L [y]$.

    A natural attempt to fix this failure for $\Kff{3}{1}$ is to replace the premise ``$y \models \phi \lor \Dia \phi$'' by ``$y \models \phi \lor \Dia \Dia \phi$''. However, this modified definition may not be well-defined: if $\Dia \phi \in \Theta$ but $\Dia \Dia \phi \notin \Theta$, then $[y] = [y']$ does not imply that $y$ and $y'$ agree on $\Dia \Dia \phi$.
    
    One might further try to extend $\Theta$ to $\Theta' = \{\Dia \Dia \phi: \Dia \phi \in \Theta\}$, while the issue persists with a formula $\Dia \Dia \phi \in \Theta'$ such that $\Dia \Dia \Dia \phi \not \in \Theta$. So, this approach only leads to infinitely many extensions and an infinite set of formulas, so that the resulting filtration may no longer be finite. From these observations, we can see that the problem comes from the presupposition that we use the same set of formulas to define the equivalence relation and the accessibility relation between the equivalence classes. The idea of definable filtration is exactly to remove this constraint, which is unnecessary for many purposes.
    
\end{remark}

We present an algebraic definable filtration construction for $\Kff{m+1}{1}$ ($m \geq 1$), meaning that when applied to a $\Kff{m+1}{1}$-algebra, the filtrated algebra is also a $\Kff{m+1}{1}$-algebra. This filtration will be used to develop the theory of stable canonical formulas for pre-transitive logics $\Kff{m+1}{1}$ in \Cref{Sec 5}. The construction is the algebraic dual of the frame-theoretic filtration presented in the proof of \cite[Theorem 8]{ageneralfiltrationmethod1972}. 

\begin{lemma} \label{3: Lem K4m1 filtration}
    Let $\A = (A, \Dia)$ be a $\Kff{m+1}{1}$-algebra, $V$ be a valuation on $\A$, and $\Theta$ be a finite subformula-closed set of formulas. Let $\Theta' = \Sub(\Theta \cup \{\Dia^m\phi: \phi \in \Theta\})$ and $A'$ be the Boolean subalgebra of $A$ generated by $V[\Theta']$. Define the modal oparators $\Dia_0$ and $\Dia_1$ on $A'$ by
    \[\Dia_0 a = \bigwedge \{b \in A': \Dia a \leq b\} \text{ and } \Dia_1a = \bigvee \{\Dia_0^{km+1}a: k \in \omega\}. \]
    Then, $\A' = (A', \Dia_1)$ is a definable filtration of $\A$ for $\Theta$ through $\Theta'$ and $\A' \models \Kff{m+1}{1}$. 
\end{lemma}

\begin{proof}
    Note that $\Theta'$ is a finite subformula-closed set, so $A'$ is finite since Boolean algebras are locally finite. For each $a \in A'$, since $A'$ is finite, there exists $k_a \in \omega$ such that $\{\Dia_0^{km+1}a: k \leq k_a\} = \{\Dia_0^{km+1}a: k \in \omega\}$. Let $K = \max\{k_a: a \in A'\}$, which also exists because $A'$ is finite. Then, $\Dia_1a = \bigvee \{\Dia_0^{km+1}a: k \leq K\}$ for any $a \in A'$, so $\Dia_1$ is always a finite join and well-defined. In fact, $\Dia_1a = \bigvee \{\Dia_0^{km+1}a: k \leq K'\}$ for any $K' \geq K$.

    Let $D' = \{V(\phi): \Dia \phi \in \Theta'\}$. We know from \Cref{3: Ex filtration} that $(A', \Dia_0)$ is the least filtration of $\A$ through $\Theta'$, so $i: (A', \Dia_0) \inj_{D'} (A, \Dia)$, where $i$ is the inclusion map. Since $\Dia_0 0 = 0$, we have $\Dia_10 = 0$. Since $A'$ is closed under finite joins and $\Dia_0$ preserves them, we have 
    \begin{align*}
        \Dia_1a \lor \Dia_1b &= \bigvee \{\Dia_0^{km+1}a: k \leq K\} \lor \bigvee \{\Dia_0^{km+1}b: k \leq K\} \\
        &= \bigvee \{\Dia_0^{km+1}a \lor \Dia_0^{km+1}b: k \leq K\} \\
        &= \bigvee \{\Dia_0^{km+1}(a \lor b): k \leq K\} \\
        &= \Dia_1(a \lor b).
    \end{align*}
    So, $\Dia_1$ fixes 0 and commutes with $\lor$, hence $(A', \Dia_1)$ is a modal algebra. Next, we show by induction that $\Dia_1^la = \bigvee \{\Dia_0^{km+l}a: k \leq K\}$ for $l \geq 1$. This holds for $l=1$ by the definition of $\Dia_1$. Assuming that it holds for $l$, we have
    \begin{align*}
        \Dia_1^{l+1}a &= \Dia_1\Dia_1^la \\
        &= \bigvee \{\Dia_0^{km+1}\Dia_1^l a: k \leq K\} \\
        &= \bigvee \{\Dia_0^{km+1}\bigvee \{\Dia_0^{k'm+l}a: k' \leq K\}: k \leq K\} \\
        &= \bigvee \{\bigvee \{\Dia_0^{km+1}\Dia_0^{k'm+l}a: k' \leq K\}: k \leq K\} \\
        &= \bigvee \{\Dia_0^{(k+k')m+l+1} a: k, k' \leq K\} \\
        &= \bigvee \{\Dia_0^{km+l+1} a: k \leq 2K\} \\
        &= \bigvee \{\Dia_0^{km+l+1} a: k \leq K\},
    \end{align*}
    showing that the statement also holds for $l+1$. So, $\Dia_1^la = \bigvee \{\Dia_0^{km+l}a: k \leq K\}$ holds for $l \geq 1$, and we obtain that for any $a \in A'$, 
    \begin{align*}
        \Dia_1^{m+1}a &= \bigvee \{\Dia_0^{km+m+1}a: k \leq K\} \\
        &= \bigvee \{\Dia_0^{(k+1)m+1}a: k \leq K\} \\
        &\leq \bigvee \{\Dia_0^{km+1}a: k \leq K+1\} \\
        &= \bigvee \{\Dia_0^{km+1} a: k \leq K\} \\
        &= \Dia_1a.
    \end{align*}
    Thus, $(A', \Dia_1)$ is a $\Kff{m+1}{1}$-algebra. 

    Let $D = \{V(\phi): \Dia \phi \in \Theta\}$. It remains to show that $i': (A', \Dia_1) \inj_D (A, \Dia)$, where $i' = i$ as a map. Note that $i'$ is a Boolean embedding. Since $i$ is stable and $\Dia_0 a \leq \Dia_1 a$ by definition, we have $\Dia i'(a) = \Dia i(a) \leq i(\Dia a) = i'(\Dia_0a) \leq i'(\Dia_1a)$. So, $i'$ is stable. 

    Let $d \in D$. Then, $d = V(\phi)$ for some $\Dia \phi \in \Theta$, so $\Dia^{m+1} \phi, \dots, \Dia \phi \in \Theta'$ and $\Dia^md, \dots, d \in D'$. Moreover, since $\Dia_0$ is the least filtration, we have $\Dia^md = \Dia_0^md, \dots, \Dia d = \Dia_0 d$ (see, e.g., \cite[Lemma 4.5]{stablecanonicalrules}). Since $i$ satisfies CDC for $D'$ and $d \in D'$, we have $i'(\Dia_0d) \leq \Dia i'(d)$. Assume that $i'(\Dia_0^{m'}d) \leq \Dia^{m'} i'(d)$ for some $ 1 \leq m' \leq m$. Then $\Dia i'(\Dia_0^{m'}d) \leq \Dia^{m'+1} i'(d)$. Again since $i$ satisfies CDC for $D'$ and $\Dia_0^{m'}d = \Dia^{m'} d \in D'$, we have $i'(\Dia_0^{m'+1}d) \leq \Dia i'(\Dia_0^{m'}d)$, thus $i'(\Dia_0^{m'+1}d) \leq \Dia^{m'+1} i'(d)$. Inductively, we obtain $i'(\Dia_0^{m+1}d) \leq \Dia^{m+1} i'(d)$. Morover, since $(A, \Dia)$ is a $\Kff{m+1}{1}$-algebra and $i$ is stable, we have $\Dia^{m+1} i'(d) \leq \Dia i'(d) \leq i'(\Dia_0d)$. So, $i'(\Dia_0^{m+1}d) \leq i'(\Dia_0d)$, and $\Dia_0^{m+1}d \leq \Dia_0d$ since $i'$ is a Boolean embedding. Since $\Dia_0$ is monotone, we inductively obtain $\Dia_0^{km+1}d \leq \Dia_0d$ for all $k \geq 1$. So,
    \[i'(\Dia_1d) = i'(\bigvee \{\Dia_0^{km+1}d: k \leq K\}) \leq i'(\Dia_0d) \leq \Dia i'(d).\]
    Thus, $i'$ satisfies CDC for $D$. Hence, $i': (A', \Dia_1) \inj_D (A, \Dia)$ and $(A', \Dia_1)$ is a definable filtration of $(A, \Dia)$ for $\Theta$ through $\Theta'$.  
\end{proof}

\begin{theorem} \label{3: Thm K4m1 admits filtration}
    For any $m \geq 1$, the logic $\Kff{m+1}{1}$ admits definable filtration. 
\end{theorem}

\begin{proof}
    This follows from \Cref{3: Lem K4m1 filtration}. Note that $\Theta'$ defined in \Cref{3: Lem K4m1 filtration} does not depend on $\A$.
\end{proof}

\begin{corollary}[\cite{ageneralfiltrationmethod1972}] \label{3: Cor K4m1 fmp}
    For any $m \geq 1$, the logic $\Kff{m+1}{1}$ has the fmp.
\end{corollary}

\begin{proof}
    This follows from \Cref{3: Thm K4m1 admits filtration} and \Cref{3: Prop filtration implies fmp}.
\end{proof}

\section{Stable canonical formulas for pre-transitive logics} \label{Sec 5}

One of the reasons that rules are more powerful than formulas is that there is a hidden universal quantifier in the semantical interpretation of rules. This allows stable canonical rules to have better control over the structure of modal algebras and enables us to prove the semantic characterization \Cref{3: Thm sc rule char}. However, if the \emph{master modality} (see, e.g., \cite[Section 6.5]{blackburnModalLogic2001}) is ``definable'' in the base logic, then we can turn stable canonical rules into formulas. This idea is realized in \cite[Section 6]{stablecanonicalrules} for $\Kf$, where the theory of \emph{stable canonical formulas} for $\Kf$ is developed as an alternative to Zakharyaschev's canonical formulas \cite{zakharyaschevCanonicalFormulas1}.

In this section, we develop the theory of stable canonical formulas for pre-transitive logics $\Kff{m+1}{1}$, generalizing the stable canonical formulas for $\Kf$. We already showed that these logics admit definable filtration (\Cref{3: Thm K4m1 admits filtration}). As we can observe in \Cref{3: Def scf for pretran}, the master modality is also definable in $\Kff{m+1}{1}$ in a similar manner as in $\Kf$. As applications, we obtain the fmp for $\Kff{m+1}{1}$-stable logics and an axiomatic characterization of splittings and union-splittings in the lattice $\NExt{\Kff{m+1}{1}}$. The former yields large classes of logics with the fmp that are substantially different from similar classes studied in the literature. For each $m \geq 2$, there are continuum many $\Kff{m+1}{1}$-stable logics that are neither $\Kf$-stable logics nor subframe logics (\Cref{Thm conti many K4m1 stable}).

As we are dealing with logics in this section, it is useful to work with s.i.~modal algebras, as every variety is generated by its s.i.~members (see, e.g., \cite[Corollary 9.7]{ACourseInUniversalAlgebra1981}). We recall the following lemma from \cite[Lemma 6.4]{analgebraicaproachtocanonicalformulas}, which is proved as a corollary of Venema’s characterization \cite{venemaDualCharSI2004} of s.i.~modal algebras. 

\begin{lemma} \label{3: Lem si reflect stable subalg}
    Let $\A$ be a finite modal algebra and $\B$ be a s.i.~modal algebra. If $\A$ is a stable subalgebra of $\B$, i.e., $\A \inj_\emp \B$, then $\A$ is also s.i.
\end{lemma}

We first show how to construct finite refutation patterns for modal formulas. The proof is similar to \Cref{3: Thm scr complete}.

\begin{lemma} \label{3: Lem K4m1 refutation pattern}
    For any formula $\phi$, there exist pairs $(\A_1, D_1), \dots, (\A_n, D_n)$ such that each $\A_i = (A_i, \Dia_i)$ is a finite s.i.~$\Kff{m+1}{1}$-algebra, $D_i \subseteq A_i$, and for any s.i.~modal algebra $\B = (B, \Dia)$, the following conditions are equivalent:
    \begin{enumerate}
        \item[(1)] $\B \not\models \phi$.
        \item[(2)] There is $1 \leq i \leq n$ and a stable embedding $h: \A_i \inj_{D_i} \B$.
        \item[(3)] There is a s.i.~homomorphic image $\C = (C, \Dia)$ of $\B$, $1 \leq i \leq n$, and a stable embedding $h: \A_i \inj_{D_i} \C$.
    \end{enumerate}
\end{lemma}

\begin{proof}
    If $\phi \in \Kff{m+1}{1}$, then let $n=0$. Assume that $\phi \notin \Kff{m+1}{1}$ and let $\Theta = \Sub(\phi)$. Since $\Kff{m+1}{1}$ admits definable filtration by \Cref{3: Thm K4m1 admits filtration}, there is a finite subformula-closed set $\Theta'$ containing $\Theta$ such that, for any $\Kff{m+1}{1}$-algebra $\B$ and any valuation $V$ on $\B$, there is a definable filtration $(\B', V')$ of $(\B, V)$ for $\Theta$ through $\Theta'$ such that $\B' \models \Kff{m+1}{1}$. Let $l = |\Theta'|$. Since Boolean algebras are locally finite, up to isomorphism, there are finitely many tuples $(\A, V, D)$ satisfying the following conditions:
    \begin{enumerate}
        \item $\A$ is a finite s.i.~$\Kff{m+1}{1}$-algebra based on an at most $l$-generated Boolean algebra and $\A \not\models \phi$,
        \item $V$ is a valuation on $\A$ such that $\A, V \not\models \phi$ and $V(p) = 0$ for $p \notin \Theta'$,
        \item $D = \{V(\psi): \Dia\psi \in \Theta\}$.
    \end{enumerate}
    Let $(\A_1, V_1, D_1), \dots, (\A_n, V_n, D_n)$ be an enumeration of such tuples. We show that $(\A_1, D_1), \dots, (\A_n, D_n)$ is the desired set of pairs. Let $\B$ be a s.i.~$\Kff{m+1}{1}$-algebra. 

    $(1) \Rightarrow (2)$. Suppose that $\B \not\models \phi$. Let $V$ be a valuation on $\B$ such that $\B, V \not\models \phi$. Then there is a definable filtration $(\B', V')$ of $(B, V)$ for $\Theta$ through $\Theta'$ such that $\B' \models \Kff{m+1}{1}$. By the definition of definable filtration, $\B'$ is a stable subalgebra of $\B$, so $\B'$ is also s.i.~by \Cref{3: Lem si reflect stable subalg}. Then the same argument as in the proof of \Cref{3: Thm scr complete} shows that the tuple $(\B', V', D)$ is identical to $(\A_i, V_i, D_i)$ for some $1 \leq i \leq n$. Since $\B' \inj_D \B$ by the definition of definable filtration, we conclude $\A_i \inj_{D_i} \B$.

    $(2) \Rightarrow (3)$. This is obvious by taking $\C = \B$.

    $(3) \Rightarrow (1)$. Suppose that there is a s.i.~homomorphic image $\C$ of $\B$, $1 \leq i \leq n$, and a stable embedding $h: \A_i \inj_{D_i} \C$. Let $V_i$ be valuation on $\A_i$ such that $\A_i, V_i \not\models \phi$. Define a valuation $V$ on $\C$ by $V(p) = h(V_i(p))$. The same argument as in the proof of \Cref{3: Thm scr complete} shows that $\C \not\models \phi$. Thus, $\B \not\models \phi$ since $\C$ is a homomorphic image of $\B$. 
    
\end{proof}

Now we define \emph{stable canonical formulas} for pre-transitive logics $\Kff{m+1}{1}$, which capture the semantic condition (3) in the theorem above. 

\begin{definition} \label{3: Def scf for pretran}
    Let $\A$ be a finite s.i.~$\Kff{m+1}{1}$-algebra and $D \subseteq A$. Let $\rho(\A, D) = \Gamma / \Delta$ be the stable canonical rule defined in \Cref{3: Def scr}. We define the \emph{stable canonical formula} $\gamma^m(\A, D)$ as 
    \begin{align*}
        \gamma^m(A, D) &= \bigwedge\{\Box^{\leq m} \gamma: \gamma \in \Gamma\} \to \bigvee\{\Box^{\leq m}\delta: \delta \in \Delta\} \\
        &= \Box^{\leq m} \bigwedge \Gamma \to \bigvee\{\Box^{\leq m}\delta: \delta \in \Delta\},
    \end{align*}
    where $\Box^m \phi = \Box \cdots \Box \phi$ ($m$ times $\Box$) and $\Box^{\leq m} \phi = \Box^m \phi \land \dots \land \Box \phi \land \phi$.
\end{definition}

We will write $\gamma(\A, D)$ for $\gamma^1(\A, D)$; this notation is consistent with \cite{stablecanonicalrules}. The following lemma is a straightforward generalization of \cite[Lemma 4.1]{analgebraicaproachtocanonicalformulas} to pre-transitive logics $\Kff{m+1}{1}$. Recall that a filter $F$ on a modal algebra is a \emph{$\Box$-filter} if $\Box a \in F$ for each $a \in F$.

\begin{lemma} \label{3: Lem si auxiliary}
    Let $\A$ be a $\Kff{m+1}{1}$-algebra and $a, b \in \A$ such that $\Box^{\leq m} a \not\leq b$. Then there exists a s.i.~$\Kff{m+1}{1}$-algebra $\B$ and a surjective homomorphism $f: \A \to \B$ such that $f(\Box^{\leq m} a) = 1$ and $f(b) \neq 1$.
\end{lemma}

\begin{proof}
    Let $\A$ be a $\Kff{m+1}{1}$-algebra and $a, b \in \A$ such that $\Box^{\leq m} a \not\leq b$. Let $F = \up \Box^{\leq m} a$. Then $F$ is a filter on $\A$. If $c \in F$, i.e., $c \geq \Box^{\leq m} a$ for some $c \in \A$, then since $\Box a \leq \Box^{m+1} a$ by $\A \models \Kff{m+1}{1}$, we have
    \[\Box c \geq \Box (\Box^{\leq m} a) = \Box a \land \cdots \land \Box^{m+1} a \geq \Box a \land \cdots \land \Box^{m} a \geq \Box^{\leq m} a,\]
    thus $\Box c \in F$. So, $F$ is a $\Box$-filter. By Zorn's Lemma, there is a $\Box$-filter $M$ such that $\Box^{\leq m} a \in M$ and $b \notin M$, and $M$ is maximal in this sense. 

    Let $B = A/\sim$, where $a \sim b$ iff $a \leftrightarrow b\in M$ for $a, b \in A$. Since $M$ is a $\Box$-filter, operations on $\A$ induce operations on $B$ and turn it into a modal algebra $\B$. Let $f: A \to B$ be the quotient map. It is straightforward to verify that $f$ is a surjective modal algebra homomorphism. Moreover, $f(\Box^{\leq m} a) = 1$ and $f(b) \neq 1$ since $\Box^{\leq m} a \in M$ and $b \notin M$. 

    By the correspondence between $\Box$-filters of $\B$ and $\Box$-filters of $\A$ containing $M$, if $F'$ is a $\Box$-filter on $\B$ such that $\{1\} \subsetneq F'$, then its corresponding $\Box$-filter on $\A$ contains $b$ by the maximality of $M$, so $f(b) \in F'$. It follows that $\Box^{\leq m} f(b) \in F'$. So, $\uparrow \Box^{\leq m} f(b)$ is the smallest $\Box$-filter of $\B$ properly containing $\{1\}$. Therefore, $B$ is s.i.
\end{proof}

\begin{theorem} \label{3: Thm scf char}
    Let $\A$ be a finite s.i.~$\Kff{m+1}{1}$-algebra and $D \subseteq A$. Then, for any  $\Kff{m+1}{1}$-algebra $\B$, 
    \[\B \not\models \gamma^m(\A, D) \text{ iff there is a s.i.~homomorphic image $\C$ of $\B$ such that $\A \inj_D \C$}.\]
\end{theorem}

\begin{proof}
    Suppose that there is a s.i.~homomorphic image $\C$ of $\B$ and a stable embedding $h: \A \inj_D \C$. Define a valuation $V_A$ on $\A$ by $V_A(p_a) = a$. It follows from the definition of $\Gamma$ and $\Delta$ (\Cref{3: Def scr}) that $V_A(\gamma) = 1$ for all $\gamma \in \Gamma$ and $V_A(\delta) \neq 1$ for all $\delta \in \Delta$. Thus, $V_A(\Boxx{\leq m} \Land \Gamma) = 1$ and $V_A(\Boxx{\leq m} \delta) = \Boxx{\leq m} V_A(\delta) \neq 1$ for all $\delta \in \Delta$. Since $\A$ is s.i., $\A$ has an opremum $c$ by \Cref{2: Prop opremum}. So, for each $\delta \in \Delta$, there is $n \in \omega$ such that $\Box^n (\Boxx{\leq m} V_A(\delta)) \leq c$. Since $\A \models \Kff{m+1}{1}$, for any $a \in \A$ and $k \in \omega$, $\Box^{k'} a \leq \Box^{k} a$ for some $0 \leq k' \leq m$. Thus, we have $\Boxx{\leq m} V_A(\delta) \leq \Box^n (\Boxx{\leq m} V_A(\delta)) \leq c$. Hence, $\Lor \Boxx{\leq m} V_A(\delta) \leq c$, which implies $\A \not\models \gamma^m(\A, D)$. Next, define a valuation $V_C$ on $\C$ by $V_C(p_a) = h(V_A(p_a))$. The same argument as in the proof of \Cref{3: Thm scr complete} shows that $V_C(\gamma) = 1$ for all $\gamma \in \Gamma$ and $V_C(\delta) \neq 1$ for all $\delta \in \Delta$. Thus, $V_C(\Boxx{\leq m} \Land \Gamma) = 1$ and $V_C(\Boxx{\leq m} \delta) = \Boxx{\leq m} V_C(\delta) \neq 1$ for all $\delta \in \Delta$. Since $\C$ is a homomorphic image of $\B$ and $\B \models \Kff{m+1}{1}$, $\C \models \Kff{m+1}{1}$ as well. Since $\C$ is also s.i., applying the same argument as for $\A$, we obtain that $V_C(\Boxx{\leq m} \Land \Gamma) = 1$ and $\Lor \Boxx{\leq m} V_C(\delta)$ is below or equal to the opremum of $\C$, so $\C \not\models \gamma^m(\A, D)$. It follows that $\B \not\models \gamma^m(\A, D)$ since $\C$ is a homomorphic image of $\B$.

    Conversely, suppose that $\B \not\models \gamma^m(\A, D)$. Then there is a valuation $V_B$ on $\B$ such that $\Boxx{\leq m} V_B(\Land \Gamma) \not\leq V_B(\Lor \Boxx{\leq m} \delta)$. By \Cref{3: Lem si auxiliary}, there is a s.i.~homomorphic image $\C$ of $\B$ and a valuation $V_C$ on $\C$ such that $V_C(\Boxx{\leq m} \Land \Gamma) = 1$ and $V_C(\Lor \Boxx{\leq m} \delta) \neq 1$. Define a map $h: \A \to \C$ by $h(a) = V_C(p_a)$. Unfolding the definition of $\Gamma$ and $\Delta$ (\Cref{3: Def scr}), it is straightforward to verify that $h$ is a stable embedding satisfying CDC for $D$. 
\end{proof}

Combining \Cref{3: Lem K4m1 refutation pattern} and \Cref{3: Thm scf char}, we obtain the following theorem, which is a version of \Cref{3: Thm scr complete} for stable canonical formulas for pre-transitive logics.

\begin{theorem} \label{3: Thm scf complete}
    For any formula $\phi$, there exist stable canonical formulas $\gamma^m(\A_1, D_1), \dots$, $\gamma^m(\A_n, D_n)$ where each $\A_i$ is a finite s.i.~$\Kff{m+1}{1}$-algebra and $D_i \subseteq A_i$, such that for any s.i.~modal algebra $\B$, 
    \[\B \models \phi \text{ iff } \B \models \Land \{\gamma^m(\A_i, D_i): 1 \leq i \leq n\}.\]
\end{theorem}

\begin{proof}
    This follows directly from \Cref{3: Lem K4m1 refutation pattern} and \Cref{3: Thm scf char}.
\end{proof}

Now we arrive at the axiomatization result for logics above $\Kff{m+1}{1}$, generalizing the result for logics above $\Kf$ (\cite[Theorem 6.10]{stablecanonicalrules}).

\begin{theorem} \label{3: Thm scf complete 2}
    Let $m \geq 1$. Every logic $L \supseteq \Kff{m+1}{1}$ is axiomatizable over $\Kff{m+1}{1}$ by stable canonical formulas. Moreover, if $L$ is finitely axiomatizable over $\Kff{m+1}{1}$, then $L$ is axiomatizable over $\Kff{m+1}{1}$ by finitely many stable canonical formulas.
\end{theorem}

\begin{proof}
    Let $L$ be a logic containing $\Kff{m+1}{1}$. Then $L = \Kff{m+1}{1} + \{\phi_i: i \in I\}$ for a set $\{\phi_i: i \in I\}$ of formulas. By \Cref{3: Thm scf complete}, for each formula $\phi_i$, there exists a finite set of stable canonical formulas $\{\gamma^m(\A_{ij}, D_{ij}): 1 \leq j \leq n_i\}$ such that for any s.i.~$\Kff{m+1}{1}$-algebra $\B$, $\B \models \phi$ iff $\B \models \{\gamma^m(\A_{ij}, D_{ij}): 1 \leq j \leq n_i\}$. So, for any s.i.~$\Kff{m+1}{1}$-algebra $\B$, $\B \models L$ iff $\B \models \gamma^m(\A_{ij}, D_{ij})$ for all $i \in I$ and $1 \leq j \leq n_i$. Since every logic is determined by the class of its s.i.~modal algebras, it follows that $L = \Kff{m+1}{1} + \{\gamma^m(\A_{ij}, D_{ij}): i \in I, 1 \leq j \leq n_i\}$. Moreover, if $L$ is finitely axiomatizable over $\Kff{m+1}{1}$, then we can choose $I$ to be finite, hence the set $\{\gamma^m(\A_{ij}, D_{ij}): i \in I, 1 \leq j \leq n_i\}$ is also finite.  
\end{proof}

\begin{remark} \label{3: Rem before jankov fml}
    We have a computability result similar to \Cref{3: Rem scr computable}. Since the logic $\Kff{m+1}{1}$ is finitely axiomatizable and has the fmp (\Cref{3: Cor K4m1 fmp}), it is decidable. Thus, the enumeration in the proof of \Cref{3: Lem K4m1 refutation pattern} is computable. Therefore, the result of \Cref{3: Thm scf complete 2} is computable; that is, there is an algorithm that, given a finite axiomatization of $L$ over $\Kff{m+1}{1}$, computes a finite set of stable canonical formulas that axiomatize $L$ over $\Kff{m+1}{1}$. 
\end{remark}

Finally, we discuss two extreme types of stable canonical formulas, namely, stable canonical formulas $\gamma^m(\A, D)$ with $D = \emp$ and those with $D = A$. 

\begin{definition} \label{3: Def stable and Jankov formulas}
    Let $\A$ be a finite modal algebra.
    \begin{enumerate}
        \item A stable canonical formula $\gamma^m(\A, \emp)$ is called a \emph{stable formula}.
        \item A stable canonical formula $\gamma^m(\A, A)$ is called a \emph{Jankov formula}.
    \end{enumerate}
\end{definition}

It follows immediately from \Cref{3: Thm scf char} that for any $\Kff{m+1}{1}$-algebra $\B$,
\[\B \not\models \gamma^m(\A, \emp) \text{ iff $\A$ is a stable subalgebra of a s.i.~homomorphic image of $\B$}\]
and
\[\B \not\models \gamma^m(\A, A) \text{ iff $\A$ is a subalgebra of a s.i.~homomorphic image of $\B$.}\]
The terminology is motivated by the fact that this semantic characterization coincides with the original \emph{Jankov formula} \cite{jankov1963} (see also \cite{dejongh1968}) for Heyting algebras. It turns out that stable formulas and Jankov formulas axiomatize logics with special features.

\begin{definition} \label{3: Def M-stable logic} 
    Let $\logic{M}$ be a logic.
    \begin{enumerate}
        \item A class $\class{K}$ of $\logic{M}$-agebras is \emph{$\logic{M}$-stable} if for any $\logic{M}$-algebra $\A$ and any $\B \in \class{K}$, if $\A \inj_\emp \B$, then $\A \in \class{K}$. 
        \item A logic $L \supseteq \logic{M}$ is \emph{$\logic{M}$-stable} if $\V(L)$ is generated by an $\logic{M}$-stable class. 
    \end{enumerate}
\end{definition}

The notion of $\logic{M}$-stable logics is introduced and studied in \cite{stablemodallogic} (see also \cite{FiltrationRevisitedLattices2018} for a comprehensive account). It was shown that each $\Kf$-stable logic is axiomatized by stable formulas over $\Kf$. This result readily generalizes to pre-transitive logics $\Kff{m+1}{1}$. We only sketch the proof as it is very similar to the one for $\Kf$.

\begin{proposition} \label{4: Prop finitely K4m1 stable}
    Let $m \geq 1$. Every $\Kff{m+1}{1}$-stable logic is axiomatizable by stable formulas over $\Kff{m+1}{1}$. 
\end{proposition}

\begin{proof}[Proof Sketch]
    Since $\Kff{m+1}{1}$ admits definable filtration, following the proof of \cite[Theorem 3.8]{stablemodallogic}, we can show that $\V(L)$ is generated by a $\Kff{m+1}{1}$-stable class $\class{K}$ of finite $\Kff{m+1}{1}$-algebras. Also, \cite[Lemma 4.5]{stablemodallogic} holds for $\Kff{m+1}{1}$ because for any finite $\Kff{m+1}{1}$-space $\F$, adding an extra root to $\F$ so that it sees itself, as well as every point in $\F$, results in another $\Kff{m+1}{1}$-space. Thus, following the proof of  \cite[Theorem 4.7]{stablemodallogic} and using \Cref{3: Lem si reflect stable subalg}, we obtain the statement. 
\end{proof}

Note that the converse does not hold. A logic axiomatized by a stable formula over $\Kf$ that is not $\Kf$-stable can be found in \cite[Example 4.11]{stablemodallogic}.

Moreover, the fmp result regarding $\logic{M}$-stable logics \cite[Proposition 3.3]{stablemodallogic} generalizes as follows.

\begin{theorem} \label{3: Thm M-stable fmp}
    Let $\logic{M}$ be a logic that admits definable filtration. Then, every $\logic{M}$-stable logic has the fmp. 
\end{theorem}

\begin{proof}
    Let $L$ be an $\logic{M}$-stable logic, that is, $\V(L)$ is generated by an $\logic{M}$-stable class $\class{K}$. Then, for any $\phi \notin L$, there is some algebra $\A \in \class{K}$ such that $\A \not\models \phi$, witnessed by some valuation $V$ on $\A$. Note that $\A$ is an $L$-algebra as $\class{K} \subseteq \V(L)$, and thus an $\logic{M}$-algebra since $\logic{M} \subseteq L$. By the assumption that $\logic{M}$ admits definable filtration, there is a finite subformula-closed set $\Theta'$ containing $\Sub(\phi)$ and a definable filtration $(\A', V')$ of $(\A, V)$ for $\Sub(\phi)$ through $\Theta'$ such that $\A' \in \V(\logic{M})$. Since $\A' \inj_\emp \A$ by the definition of definable filtrations and $\A \in \class{K}$, we have $\A' \in \class{K}$ by $\class{K}$ being $\logic{M}$-stable, and $\A' \in \V(L)$ since $\class{K} \subseteq \V(L)$. Thus, $\A'$ is a finite $L$-algebra that refutes $\phi$. Hence, $L$ has the fmp.
\end{proof}

\begin{corollary} \label{3: Cor K4m1-stable fmp}
    For $m \geq 1$, every $\Kff{m+1}{1}$-stable logic has the fmp.
\end{corollary}

\begin{proof}
    This follows from \Cref{3: Thm K4m1 admits filtration} and \Cref{3: Thm M-stable fmp}.
\end{proof}

\begin{remark}
    By a similar argument as the proof of \Cref{3: Thm M-stable fmp}, one can also generalize the characterization results in \cite{stablemodallogic} (e.g., Theorem 3.8 and Theorem 4.7), replacing filtration with definable filtration and $\Kf$ with $\Kff{m+1}{1}$. These characterizations and \Cref{4: Prop finitely K4m1 stable} are used in \cite{takahashi2025choppingfinelyfinitecountermodels} (see also \cite[Chapter 4]{TakahashiThesis}) to obtain a general fmp result, including the fmp of all $\Kff{m+1}{1}$-stable logics as stated in \Cref{3: Cor K4m1-stable fmp}.
\end{remark}

As we will see below, $\Kff{m+1}{1}$-stable logics form large classes of logics with the fmp. More precisely, we show that, for each $m \geq 2$, there are continuum many $\Kff{m+1}{1}$-stable logics that do not contain $\Kff{m'+1}{1}$ for all $1 \leq m' < m$. So, in particular, these $\Kff{m+1}{1}$-stable logics are not $\Kff{m'+1}{1}$-stable for all $1 \leq m' < m$. Note that \cite[Theorem 5.6]{stablemodallogic} proved that there are continuum many $\logic{S4}$-stable logics, and continuum many $\Kf$-stable logics between $\Kf$ and $\logic{S4}$.

Subframe logics were introduced by Fine \cite{fineK4II} for transitive logics, and their fmp has been extensively studied (see, e.g., \cite[Section 11.3]{czModalLogic1997}). A \emph{subframe} $S$ of a modal space $\X$ is a clopen subset $S \subseteq X$ with the restricted topology and relation. A logic is a \emph{subframe logic} if its class of modal spaces is closed under subframes. It follows from \cite[Theorem 39]{TwoTypesFiltrations2025} that every subframe canonical extension of $\Kff{m+1}{1}$ has the fmp for $m \geq 1$. However, the $\Kff{m+1}{1}$-stable logics we present below are not subframe logics, and thus do not fall in the scope of this result.

We need an auxiliary lemma, which states a partial converse of \Cref{4: Prop finitely K4m1 stable}. For a finite rooted $\Kff{m+1}{1}$-space $\F$, let us write $\gamma^m(\F, \emp)$ for $\gamma^m(\F^*, \emp)$ for simplicity. 

\begin{lemma} \label{Lem K4m1 stable ref root}
    Let $L = \Kff{m+1}{1} + \{\gamma^m(\F_i, \emp): i \in I\}$, where each $\F_i$ is a finite rooted $\Kff{m+1}{1}$-space with a reflexive root $r$. Then $L$ is $\Kff{m+1}{1}$-stable.
\end{lemma}

\begin{proof} 
    It suffices to show that the variety $\V(L)$ is generated by the class $\V(L)\fsi$ of finite s.i.~$L$-algebras and that $\V(L)\fsi$ is closed under stable $\Kff{m+1}{1}$-subalgebras. Let $\phi \notin L$. Then there is a s.i.~$L$-algebra $\B$ such that $\B \not\models \phi$. By \Cref{3: Lem K4m1 refutation pattern}, there is a finite s.i.~$\Kff{m+1}{1}$-algebra $\A$ and $D \subseteq A$ such that $\A \not\models \phi$ and $\A \inj_D \B$ (for example, by taking $(\A, D) = (\A_1, D_1)$). Suppose for a contradiction that $\A \not\models L$. Then $\A \not\models \gamma^m(\F_i, \emp)$ for some $i \in I$, which by \Cref{3: Thm scf char} means that $\F_i$ is a stable image of a topo-rooted closed upset $U$ of $\A_*$, the dual space of $\A$. Since $\A_*$ is finite by $\A$ being finite, $U$ is clopen. So, we can extend the stable surjection from $U$ to $\F_i$ to a stable surjection from $\A_*$ to $\F_i$ by sending all points outside $U$ to the reflexive root $r$ of $\F_i$. Thus, $\F_i$ is a stable image of $\A_*$, namely $(\F_i)^* \inj_\emp \A$. By composition with $\A \inj_D \B$, we have $(\F_i)^* \inj_\emp \B$, and hence $\B \not\models \gamma^m(\F_i, \emp)$ by \Cref{3: Thm scf char}, contradicting the assumption that $\B \models L$. Therefore, $\A \models L$, and thus $\V(L)$ is generated by $\V(L)\fsi$. 
    
    Next, we show that $\V(L)\fsi$ is closed under stable subalgebras. Let $\B \in \V(L)\fsi$ and $\A$ be a stable $\Kff{m+1}{1}$-subalgebra of $\B$. By \Cref{3: Lem si reflect stable subalg}, $\A$ is also s.i. The same argument as above shows that $\A \models L$, and thus $\V(L)\fsi$ is closed under stable $\Kff{m+1}{1}$-subalgebras.
\end{proof}

For $m \geq 2$ and $n \in \omega$, let $\F^m$ be the left non-transitive $\Kff{m+1}{1}$-space and $\F_n$ be the right transitive space depicted below. Spaces $\F_n$ were used in \cite{LocallyFiniteReducts2017} to construct continuum many stable superintuitionistic logics and in \cite{stablemodallogic} to construct continuum many $\Kf$-stable logics (see also \cite{FiltrationRevisitedLattices2018}).

\begin{figure}[h]
        \centering
        \begin{tikzpicture}[scale=1]
            \node (x0) at (0,0) {\( \bullet \)};
            \node (xx0) at (0.4,0) {\( x_0 \)};
            \node (x1) at (0,1) {\( \bullet \)};
            \node (xx1) at (0.4,1) {\( x_1 \)};
            \node (x2) at (0,2) {\( \vdots \)};
            \node (xm-1) at (0,3) {\( \bullet \)};
            \node (xxm-1) at (0.6,3) {\( x_{m-1} \)};
            \node (xm) at (0,4) {\( \bullet \)};
            \node (xxm) at (0.4,4) {\( x_m \)};
            \draw[->] (x0) -- (x1);
            \draw[->] (xm-1) -- (xm);
            \node (Fm) at (0,-1) {\( \F^m \)};

            \node (r) at (5,0) {\( \circ \)};
            \node (rr) at (5.4,0) {\( r \)};
            \node (y) at (4,1) {\( \circ \)};
            \node (yy) at (3.6,1) {\( y \)};
            \node (y0) at (4,2) {\( \circ \)};
            \node (yy0) at (3.6,2) {\( y_0 \)};
            \node (z0) at (6,2) {\( \circ \)};
            \node (zz0) at (6.4,2) {\( z_0 \)};
            \node (y1) at (4,3) {\( \circ \)};
            \node (yy1) at (3.6,3) {\( y_1 \)};
            \node (z1) at (6,3) {\( \circ \)};
            \node (zz1) at (6.4,3) {\( z_1 \)};
            \node (y2) at (4,4) {\( \vdots \)};
            \node (z2) at (6,4) {\( \vdots \)};
            \node (yn) at (4,5) {\( \circ \)};
            \node (yyn) at (3.6,5) {\( y_n \)};
            \node (zn) at (6,5) {\( \circ \)};
            \node (zzn) at (6.4,5) {\( z_n \)};
            \node (yn1) at (4,6) {\( \circ \)};
            \node (yyn1) at (3.4,6) {\( y_{n+1} \)};
            \node (zn1) at (6,6) {\( \circ \)};
            \node (zzn1) at (6.6,6) {\( z_{n+1} \)};
            \draw[->] (r) -- (y);
            \draw[->] (y) -- (y0);
            \draw[->] (r) -- (z0);
            \draw[->] (y0) -- (y1);
            \draw[->] (y0) -- (z1);
            \draw[->] (z0) -- (z1);
            \draw[->] (z0) -- (y1);
            \draw[->] (yn) -- (yn1);
            \draw[->] (yn) -- (zn1);
            \draw[->] (zn) -- (zn1);
            \draw[->] (zn) -- (yn1);
            \node (Fn) at (5,-1) {\( \F_n \)};
        \end{tikzpicture}
    \end{figure}

Let $\F^m_n$ be the space obained by putting $\F_n$ below $\F^m$, that is, the disjoint union of $\F^m$ and $\F_n$ with additional relations $(s, t)$ for all $s \in \F_n$ and $t \in \F^m$. It is straightforward to verify that $\F^m_n$ validates $\Kff{m+1}{1}$ but not $\Kff{m'+1}{1}$ for all $1 \leq m' < m$.

\begin{lemma} \label{Lem K4m1 stable fml ex}
    For any $m \geq 2$ and $n \neq n'$, $\F^m_n \models \gamma^m(\F^m_{n'}, \emp)$ and $\F^m_{n'} \models \gamma^m(\F^m_n, \emp)$.
\end{lemma}

\begin{proof}
    Wlog, we may assume $n < n'$. Then, a simple cardinality argument shows that $\F^m_{n'}$ is not a stable image of an upset of $\F^m_n$, so $\F^m_n \models \gamma^m(\F^m_{n'}, \emp)$ by \Cref{3: Thm scf char}. Suppose for a contradiction that there are an upset $U \subseteq \F^m_{n'}$ and a stable map $f: U \surj \F^m_n$. Then, $U$ must contain $\F^m$ and $f$ must map that part to the $\F^m$ part of $\F^m_n$ identically. Since $f$ cannot map a reflexive point to an irreflexive point, it yields a stable surjection $f'$ from $U' = U \cap \F_{n'}$ to $\F_n$. However, this is impossible by the proof of \cite[Lemma 6.12]{LocallyFiniteReducts2017}. Thus, $\F^m_n$ is not a stable image of an upset of $\F^m_{n'}$, so $\F^m_{n'} \models \gamma^m(\F^m_n, \emp)$ by \Cref{3: Thm scf char}.
\end{proof}

\begin{theorem} \label{Thm conti many K4m1 stable}
    For each $m \geq 2$, there are continuum many $\Kff{m+1}{1}$-stable logics that do not contain $\Kff{m'+1}{1}$ for all $1 \leq m' < m$. Moreover, these logics are not subframe logics.
\end{theorem}

\begin{proof}
    For each $m \geq 2$, let $\F^m_n$ be the space defined above for $n \in \omega$. For each $\{0\} \subseteq I \subseteq \omega \setminus\{1\}$, let $L_I = \Kff{m+1}{1} + \{\gamma^m(\F^m_n, \emp): i \in I\}$. Since every $\F^m_n$ has a reflexive root, each $L_I$ is $\Kff{m+1}{1}$-stable by \Cref{Lem K4m1 stable ref root}. If $I \neq I'$, then we may assume wlog that there is some $i \in I \setminus I'$, and $\F^m_i \not\models \gamma^m(\F^m_i, \emp)$ by \Cref{3: Thm scf char} and $\F^m_i \models \gamma^m(\F^m_{i'}, \emp)$ for all $i' \in I$ by \Cref{Lem K4m1 stable fml ex}, hence $L_I \neq L_{I'}$. Since there are continuum many subsets of $\omega$ with 0 and without 1, there are continuum many such logics. Moreover, since $\F^m_1 \models L_I$ by $1 \notin I$ and \Cref{Lem K4m1 stable fml ex}, and $\F^m_1 \not\models \Kff{m'+1}{1}$ for all $1 \leq m' < m$, we have $\Kff{m'+1}{1} \not\subseteq L_I$ for all $1 \leq m' < m$. Finally, while $\F^m_1 \models L_I$, we have $\F^m_0 \not\models L_I$ because $\F^m_0 \not\models \gamma^m(\F^m_0, \emp)$ by $0 \in I$ and \Cref{3: Thm scf char}. But $\F^m_0$ is a subframe of $\F^m_1$ by removing the points $y_2$ and $z_2$. This indeed yields a clopen subset of $\F^m_1$ since $\F^m_1$ is finite and every subset is clopen. Thus, we conclude that each $L_I$ is not a subframe logic.
\end{proof}

Next, we turn to Jankov formulas, which axiomatize exactly union-splittings.

\begin{definition}
    Let $X$ be a complete lattice. A \emph{splitting pair} of $X$ is a pair $(x, y)$ of elements of $X$ such that $x \not\leq y$ and for any $z \in X$, either $x \leq z$ or $z \leq y$. If $(x, y)$ is a splitting pair of $X$, we say that $x$ \emph{splits} $X$ and $y$ is a \emph{splitting} in $X$. 
\end{definition}

This notion has been an important tool in the study of lattices of logics. See \cite[Section 10.7]{czModalLogic1997} for a historical overview. 

\begin{definition} 
    Let $L_0$ and $L$ be logics.
    \begin{enumerate}
        \item $L$ is a \emph{splitting} in $\NExt{L_0}$ iff it is a lattice-theoretic splitting in the lattice $\NExt{L_0}$.
        \item $L$ is a \emph{union-splitting} in $\NExt{L_0}$ iff it is the join of a set of splittings in $\NExt{L_0}$.
    \end{enumerate}
\end{definition}

\begin{theorem}
    Let $m \geq 1$.
    \begin{enumerate}
        \item A logic $L \in \NExt{\Kff{m+1}{1}}$ is a splitting logic in $\NExt{\Kff{m+1}{1}}$ iff $L = \Kff{m+1}{1} + \gamma^m(\A, A)$ for some finite s.i.~$\Kff{m+1}{1}$-algebra $\A$.
        \item A logic $L \in \NExt{\Kff{m+1}{1}}$ is a union-splitting logic in $\NExt{\Kff{m+1}{1}}$ iff $L = \Kff{m+1}{1} + \{\gamma^m(\A_i, A_i): i \in I\}$ for a set $\{\A_i: i \in I\}$ of finite s.i.~$\Kff{m+1}{1}$-algebras.
    \end{enumerate}
\end{theorem}

\begin{proof}
    (2) is an immediate consequence of (1). To show (1), first suppose that $L = \Kff{m+1}{1} + \gamma^m(\A, A)$ for some finite s.i.~$\Kff{m+1}{1}$-algebra $\A$. We show that $(L, \Log(\A))$ is a splitting pair in $\NExt{\Kff{m+1}{1}}$. For any logic $L' \in \NExt{\Kff{m+1}{1}}$ such that $L \not\subseteq L'$, there is a modal algebra $\B$ such that $\B \models L'$ and $\B \not\models L$. Thus, $\B \not\models \gamma^m(\A, A)$, namely, $\A$ is a subalgebra of a homomorphic image of $\B$. So, $L' \subseteq \Log(\B) \subseteq \Log(\A)$. Hence, $(L, \Log(\A))$ is a splitting pair in $\NExt{\Kff{m+1}{1}}$. 
    
    Conversely, suppose that $L$ is a splitting logic in $\NExt{\Kff{m+1}{1}}$. Then by \Cref{3: Thm McKenzie split s.i.} below, since $\NExt{\Kff{m+1}{1}}$ has the fmp (\Cref{3: Cor K4m1 fmp}), $(L, \Log(\A))$ is a splitting pair in $\NExt{\Kff{m+1}{1}}$ for some finite s.i.~$\Kff{m+1}{1}$-algebra $\A$. By the argument above,$(\Kff{m+1}{1} + \gamma^m(\A, A), \Log(\A))$ is a splitting pair in $\NExt{\Kff{m+1}{1}}$. It follows that $L = \Kff{m+1}{1} + \gamma^m(\A, A)$.
\end{proof}

A similar characterization of (union-)splitting logics in $\NExt{\wKf}$ (where $\wKf = \K + \Dia\Dia p \to \Dia p \lor p$) can be found in \cite{CANONICALFORMULASWK4}. Note that both of these characterizations follow from the combination of the following two general results by McKenzie \cite{mckenzieEquationalBasesNonmodular1972} and by Rautenberg \cite{rautenbergSplittingLatticesLogics1980}; the proof of the latter uses a version of characteristic formulas.

\begin{theorem}[\cite{mckenzieEquationalBasesNonmodular1972}] \label{3: Thm McKenzie split s.i.}
    Let $L$ be a logic with the fmp. If a logic $L'$ splits $\NExt{L}$, then $L' = \Log(\A)$ for some finite s.i.~$L$-algebra $\A$.
\end{theorem}

\begin{theorem}[\cite{rautenbergSplittingLatticesLogics1980}]
    Let $L$ be a logic containing the logic $\K + \Dia^{m+1} p \to \Dia^m p \lor \cdots \lor p$. Then for any finite s.i.~$L$-algebra $\A$, the logic $\Log(\A)$ splits $\NExt{L}$.
\end{theorem}

\section[\texorpdfstring{The $m$-closed domain condition and $m$-stable canonical formulas}{The m-closed domain condition and m-stable canonical formulas}]{The $m$-closed domain condition and $m$-stable canonical formulas} \label{Sec 6}

A typical type of application of stable canonical formulas is to reduce a question for all formulas to all stable canonical ones, using the fact that every formula is equivalent to finitely many stable canonical formulas. Then, one can deduce results for all formulas by only working with stable canonical formulas. In this respect, if we can restrict to a smaller class of formulas that still covers all formulas, then this reduction could become even more powerful.

In this section, we consider a variant of stable canonical formulas, called $m$-stable canonical formulas, for the pre-transitive logics $\NExt{\Kff{m+1}{1}}$ ($m \geq 1$). As we will see in \Cref{Ex m-scf}, up to equivalence, these formulas form a proper subclass of the class of stable canonical formulas, while they do axiomatize all extensions of $\NExt{\Kff{m+1}{1}}$ (\Cref{3: Thm mscf complete}, cf. \Cref{3: Thm scf complete 2}). The new formulas are based on the $m$-closed domain condition, which means to preserve the modality not only one step but also up to $m$ steps. This condition seems to match better the nature of pre-transitive logics that the master modality in $\NExt{\Kff{m+1}{1}}$ is defined by $\Dia^{\leq m}$, and thus could be more handy for applications.

\begin{definition}
    Let $\A$ and $\B$ be modal algebras and $h: A \to B$ be a stable homomorphism. For $a \in A$, we say that $h$ satisfies the \emph{$m$-closed domain condition ($m$-CDC) for $a$} if $h(\Dia^k a) = \Dia^k h(a)$ for all $1 \leq k \leq m$. For $D \subseteq A$, we say that $h$ satisfies the \emph{$m$-closed domain condition ($m$-CDC) for $D$} if $h$ satisfies $m$-CDC for all $a \in D$.
\end{definition}

Similar to the case of CDC, for stable homomorphisms, $h(\Dia^k a) = \Dia^k h(a)$ holds iff $h(\Dia^k a) \leq \Dia^k h(a)$ holds. We also provide the frame-theoretic version of the definition and show that the two definitions are dual to each other.

\begin{definition} \label{3: Def mCDC sp}
    Let $\X = (X, R)$ and $\Y = (Y, Q)$ be modal spaces and $f: X \to Y$ be a stable map. For a clopen subset $D \subseteq Y$, we say that $f$ satisfies the \emph{$m$-closed domain condition ($m$-CDC) for $D$} if for all $1 \leq k \leq m$,
\[
Q^k[f(x)] \cap D \neq \emptyset \Rightarrow f(R^k[x]) \cap D \neq \emptyset.
\]
    For a set $\D$ of clopen subsets of $Y$, we say that $f: X \to Y$ satisfies the \emph{$m$-closed domain condition ($m$-CDC) for $\D$} if $f$ satisfies $m$-CDC for all $D \in \D$.
\end{definition}

\begin{proposition}
Let $\A$ and $\B$ be modal algebras and $h: A \to B$ be a stable homomorphism. For any $a \in A$, 
\[\text{$h$ satisfies $m$-CDC for $a$ iff $h_*$ satisfies $m$-CDC for $\beta(a)$.}\] 
For any $D \subseteq A$, 
\[\text{$h$ satisfies $m$-CDC for $D$ iff $h_*$ satisfies $m$-CDC for $\beta[D]$.}\]
\end{proposition}

\begin{proof}
    Let $\A$ and $\B$ be modal algebras with dual spaces $\X = (X, R)$ and $\Y = (Y, Q)$ and $h: A \to B$ be a stable homomorphism. Let $a \in A$. It follows from the duality that
    \begin{align*}
        \text{$h$ satisfies $m$-CDC for $a$} &\text{ iff } h(\Dia^k a) \leq \Dia^k h(a) \text{ for all $1 \leq k \leq m$}\\
            &\text{ iff } h_*^{-1} \Dia^k \beta(a) \subseteq \Dia^k h_*^{-1} \beta(a) \text{ for all $1 \leq k \leq m$} \\
            &\text{ iff } Q^k[h_*(x)] \cap \beta(a) \neq \emptyset \Rightarrow h_*(R^k[x]) \cap \beta(a) \neq \emp \text{ for $1 \leq k \leq m$} \\
            &\text{ iff } \text{$h_*$ satisfies $m$-CDC for $\beta(a)$.}
    \end{align*}
    The second statement follows directly from the first. 
\end{proof}

\begin{notation}
    We write $h: \A \inj^m_D \B$ if $h$ is a stable embedding satisfying $m$-CDC for $D$ and $\A \inj^m_D \B$ if there is such an $h$. We write $f: \X \surj^m_\D \Y$ if $f$ is a surjective stable map satisfying $m$-CDC for $\D$ and $\X \surj^m_\D \Y$ if there is such an $f$.
\end{notation}

Now we define \emph{$m$-stable canonical formulas} for pre-transitive logics $\Kff{m+1}{1}$ ($m \geq 1$). The only difference with standard stable canonical formulas is that in $\Gamma$, the part using the closed domain $D$ is generalized so that it captures $m$-CDC.

\begin{definition} \label{3: Def m-scf}
    Let $\A$ be a finite s.i.~$\Kff{m+1}{1}$-algebra and $D \subseteq A$. We define the \emph{$m$-stable canonical formula} $\gamma_+^m (\A, D)$ as 
    \begin{align*}
        \gamma_+^m(\A, D) &= \bigwedge\{\Box^{\leq m} \gamma: \gamma \in \Gamma\} \to \bigvee\{\Box^{\leq m}\delta: \delta \in \Delta\} \\
        &= \Box^{\leq m} \bigwedge \Gamma \to \bigvee\{\Box^{\leq m}\delta: \delta \in \Delta\},
    \end{align*}
    where 
    \begin{align*}
        \Gamma  = & \{p_a \lor p_b \leftrightarrow p_{a \lor b}: a,b \in A\} \cup \\
        & \{\lnot p_a \leftrightarrow p_{\lnot a} : a \in A\} \cup \\
        & \{\Dia p_a \rightarrow p_{\Dia a} : a \in A\} \cup \\
        & \{p_{\Dia^k a} \rightarrow \Dia^k p_a : a \in D, 1 \leq k \leq m\},
    \end{align*}
    and
    \[\Delta = \{p_a : a \in A, a \ne 1\}.\]
\end{definition}

\begin{remark}
    As stable canonical rules, $m$-stable canonical formulas can also be defined directly from finite modal spaces (i.e., finite Kripke frames). The basic idea is the same as \Cref{3: Def m-scf}, but we use $\Gamma$ and $\Delta$ in \Cref{3: Def scr sp} instead and change the last clause in $\Gamma$ to 
    \[\{p_x \to \bigvee \{\Dia^k p_y : y \in D\} : x \in X, D \in \D, x \in( R^{-1})^k[D], 1 \leq k \leq m\}\]
    in light of \Cref{3: Def mCDC sp}.
\end{remark}

\begin{remark}
    In general, $m$-CDC is stronger than the standard CDC. For $m=1$, where the base logic is $\Kf$, the $m$-CDC and $m$-stable canonical formulas reduce to the standard CDC and stable canonical formulas. It is clear that $m$-stable canonical formulas induce the same stable formulas (up to equivalence) as stable canonical formulas. Moreover, since a modal homomorphism always satisfies $m$-CDC for any $D$, they also induce the same Jankov formulas (up to equivalence).
\end{remark}

Following the proof of \Cref{3: Thm scf char}, it is straightforward to verify the semantic condition of the validity of $m$-stable canonical formulas.

\begin{theorem} \label{3: Thm mscf char}
    Let $\A$ be a finite s.i.~$\Kff{m+1}{1}$-algebra and $D \subseteq A$. Then, for any  $\Kff{m+1}{1}$-algebra $\B$, 
    \[\B \not\models \gamma_+^m(\A, D) \text{ iff there is a s.i.~homomorphic image $\C$ of $\B$ such that $\A \inj^m_D \C$}.\]
\end{theorem}

Every $m$-stable canonical formula is equivalent over $\Kff{m+1}{1}$ to a stable canonical formula, but not vice versa. Thus, up to equivalence, the class of $m$-stable canonical formulas is a proper subclass of the class of stable canonical formulas. It is easy to see from the semantic characterizations Theorems \ref{3: Thm scf char} and \ref{3: Thm mscf char} that, any $m$-stable canonical formula $\gamma_+^m(\A, D)$ is equivalent to the stable canonical formula $\gamma^m(\A, D')$, where $D' = \{\Dia^{k-1} d: d \in D, 1 \leq k \leq m\}$. The following example shows that the converse does not hold.

\begin{example} \label{Ex m-scf}
    We construct a counterexample using finite modal spaces. Let $\X$ be the finite rooted $\Kff{4}{1}$-space depicted as below and $D = \{d\}$.
    \begin{figure}[H]
        \centering
        \begin{tikzpicture}[scale=1]
            \node (1) at (0,0) {\(\bullet\)}; \node at (0.3,0) {\(a\)};
            \node (2) at (0,1) {\(\bullet\)}; \node at (0.3,1) {\(b\)};
            \node (3) at (0,2) {\(\bullet\)}; \node at (0.3,2) {\(c\)};
            \node (4) at (0,3) {\(\bullet\)}; \node at (0.3,3) {\(d\)};
            \node (x) at (0,-0.6) {\( \X\)};
            \draw[->] (1) -- (2);
            \draw[->] (2) -- (3);
            \draw[->] (3) -- (4);
        \end{tikzpicture}
        \end{figure}
    Assume for a contradiction that the stable canonical formula $\gamma^3(\X, \{D\})$ is equivalent to some $3$-stable canonical formula $\gamma_+^3(\X', \D')$. Then $\X \not\models \gamma_+^3(\X', \D')$ and $\X' \not\models \gamma^3(\X, \{D\})$. By the dual of \Cref{3: Thm scf char} and \Cref{3: Thm mscf char}, one can verify that $\X' = \X$. Let $\Y$ and $\Y'$ be rooted $\Kff{4}{1}$-spaces depicted as below.
    \begin{figure}[H]
        \centering
        \begin{tikzpicture}[scale=1]
            \node (01) at (0,0) {\( \bullet\)};
            \node (02) at (0,1) {\( \bullet\)};
            \node (03) at (0,2) {\( \bullet\)};
            \node (04) at (0,3) {\( \bullet\)};
            \node (05) at (1,1) {\( \bullet\)};
            \node (y) at (0.5,-0.6) {\( \Y\)};
            \draw[->] (01) -- (02);
            \draw[->] (02) -- (03);
            \draw[->] (03) -- (04);
            \draw[->] (01) -- (05);

            \node (11) at (6,0) {\( \bullet\)};
            \node (12) at (6,1) {\( \bullet\)};
            \node (13) at (6,2) {\( \bullet\)};
            \node (14) at (6,3) {\( \bullet\)};
            \node (15) at (7,1) {\( \bullet\)};
            \node (16) at (7,2) {\( \bullet\)};
            \node (y') at (6.5,-0.6) {\( \Y'\)};
            \draw[->] (11) -- (12);
            \draw[->] (12) -- (13);
            \draw[->] (13) -- (14);
            \draw[->] (11) -- (15);
            \draw[->] (15) -- (16);
        \end{tikzpicture}
    \end{figure}
    Note that there is only one stable map from $\Y$ or $\Y'$ to $\X$, and no proper rooted upset of $\Y$ or $\Y'$ can map onto $\X$. It is easy to see that $\Y \surj_D \X$ and $\Y' \not\surj_D \X$, thus $\Y \not\models \gamma^3(\X, \{D\})$ and $\Y' \models \gamma^3(\X, \{D\})$.
    If $d \in D'$ for some $D' \in \D'$, then $\Y \not\surj^3_{\D'} \X$, so $\Y \models \gamma_+^3(\X', \D')$. If there is no $D' \in \D'$ such that $d \in D'$, then $\Y' \surj^3_{\D'} \X$, so $\Y' \not\models \gamma_+^3(\X', \D')$. This contradicts that $\gamma^3(\X, \{D\})$ and $\gamma_+^3(\X', \D')$ are equivalent. 
\end{example}

Now we show that any logic above $\Kff{m+1}{1}$ can be axiomatized by $m$-stable canonical formulas, thus providing an alternative to \Cref{3: Thm scf complete 2}. Recall that in the case of stable canonical formulas, the essential idea was to construct finite refutation patterns as done in \Cref{3: Lem K4m1 refutation pattern}. The construction in turn essentially depends on the definable filtration developed in \Cref{3: Lem K4m1 filtration}. Thus, to adapt the whole proof for stable canonical formulas to $m$-stable canonical formulas, it suffices to observe that the definable filtration in \Cref{3: Lem K4m1 filtration} in fact induces a stable homomorphism satisfying $m$-CDC.

\begin{lemma} \label{3: Lem filtration m-CDC}
    Let $\A = (A, \Dia)$ be a $\Kff{m+1}{1}$-algebra, $V$ be a valuation on $\A$, $\Theta$ be a finite subformula-closed set of formulas, and $\Theta' = \Sub(\Theta \cup \{\Dia^m\phi: \phi \in \Theta\})$. Let $\Dia_0$ and $\Dia_1$ be the modal operators on $A'$ defined in \Cref{3: Lem K4m1 filtration}.
    Then, the inclusion $i': \A' = (A', \Dia_1) \inj_D \A$ satisfies $m$-CDC for $D$, where $D = \{V(\phi): \Dia \phi \in \Theta\}$. 
\end{lemma}

\begin{proof}
    Let $d \in D$. Following the proof of \Cref{3: Lem K4m1 filtration}, let $D' = \{V(\phi): \Dia \phi \in \Theta'\}$. Then we saw in the proof that $\Dia^md = \Dia_0^md, \dots, \Dia d = \Dia_0 d$, and these points as well as $d$ are in $\D'$. Since the inclusion $i: (A', \Dia_0) \inj_{D'} \A$ is the same map as $i'$, we have $i'(\Dia_0^ld) \leq \Dia^l i'(d)$ for all $1 \leq l \leq m$. Also, recall from the proof of \Cref{3: Lem K4m1 filtration} that $\Dia_1^la = \bigvee \{\Dia_0^{km+l}a: k \leq K\}$ for $l \geq 1$ and $a \in A'$ and $\Dia_0^{km+1}d \leq \Dia_0d$ for all $k \geq 1$. Thus, for any $1 \leq l \leq m$ we have
    \[i'(\Dia_1^l d) = i'(\Lor \{\Dia_0^{km+l} d: k \leq K\}) \leq i'(\Dia_0^l d) \leq \Dia^l i'(d).\]
    Hence, we conclude that $i'$ satisfies $m$-CDC for $D$.
\end{proof}

\begin{theorem} \label{3: Thm mscf complete}
    For any formula $\phi$, there exist stable canonical formulas $\gamma_+^m(\A_1, D_1), \dots$, $\gamma_+^m(\A_n, D_n)$ where each $\A_i$ is a finite s.i.~$\Kff{m+1}{1}$-algebra and $D_i \subseteq A_i$, such that for any s.i.~modal algebra $\B$, 
    \[\B \models \phi \text{ iff } \B \models \Land \{\gamma_+^m(\A_i, D_i): 1 \leq i \leq n\}.\]
\end{theorem}
    
\begin{proof}
    \Cref{3: Lem filtration m-CDC} shows that the condition (3) in \Cref{3: Lem K4m1 refutation pattern} can be strengthened such that the stable embedding $h: \A_i \inj_{D_i} \C$ satisfies $m$-CDC for $D_i$. Thus, combined with \Cref{3: Thm mscf char}, the statement follows.
\end{proof}

\begin{corollary} 
    Let $m \geq 1$. Any logic $L \supseteq \Kff{m+1}{1}$ is axiomatizable over $\Kff{m+1}{1}$ by $m$-stable canonical formulas. Moreover, if $L$ is finitely axiomatizable over $\Kff{m+1}{1}$, then $L$ is axiomatizable over $\Kff{m+1}{1}$ by finitely many $m$-stable canonical formulas.
\end{corollary}

\begin{proof}
    This follows from \Cref{3: Thm mscf complete} by a similar argument as in the proof of \Cref{3: Thm scf complete 2}.
\end{proof}

\section{Future work} \label{Sec 7}

As we have discussed in this paper, admitting definable filtration and being able to define the master modality are the key points to generalize the theory of stable canonical formulas. Given the success with the logics $\Kff{m+1}{1}$, one might expect a similar result for the logic $\wKf = \K + \Dia\Dia p \to p \lor \Dia p$ and other pre-transitive logics $\Kff{m}{n}$. However, even though $\wKf$ has the fmp \cite{SpectralT0SpacesDSemantics2011}, it is unknown whether it admits definable filtration (see, e.g., \cite[Section 5]{TwoTypesFiltrations2025}). Stable canonical rules work well for $\S_\K$ because rules have the master modality built into their semantics. Since the master modality is not definable in $\K$, it seems quite challenging to define stable canonical formulas for $\K$ in a meaningful way. We leave it open to generalize stable canonical formulas to non-transitive logics other than $\Kff{m+1}{1}$.

As mentioned in the introduction, a noticeable application of stable canonical rules and canonical rules is to study \emph{admissible} rules. Je\v r\'abek \cite{Jeřábek_2009} used canonical rules, and Bezhanishvili et al. \cite{admissiblebases} (see also \cite[Chapter 6]{TakahashiThesis}) used stable canonical rules to give an alternative proof of the decidability of admissibility for transitive modal logics such as $\Kf$ and $\Sf$, and the intuitionistic logic $\IPC$. We leave it for future research whether stable canonical rules or $m$-stable canonical rules can be used to study the admissibility in the pre-transitive logics $\Kff{m+1}{1}$.

\subsection*{}
\begin{acknowledgements}
I am very grateful to Nick Bezhanishvili for his supervision of the Master's thesis and his helpful comments on this paper. I would also like to thank Yde Venema for pointing out the correct reference on definable filtration to me. Finally, I would like to thank the two anonymous reviewers for their valuable comments, which significantly improved the paper, in particular, Theorem 5.17. I was supported by the Student Exchange Support Program (Graduate Scholarship for Degree Seeking Students) of the Japan Student Services Organization and the Student Award Scholarship of the Foundation for Dietary Scientific Research.
    
\end{acknowledgements}

\printbibliography[heading=bibintoc,title=References]

\end{document}